\documentclass[10pt,a4paper]{article}
\usepackage[a4paper]{geometry}
\usepackage{amssymb,latexsym,amsmath,amsfonts,amsthm}
\usepackage{graphicx,comment}
\usepackage{epsfig}
\usepackage{tikz}
\usepackage{subcaption}
\usepackage{comment}
\usepackage[T1]{fontenc}

\def\l{\lambda}

\newcommand{\supp}{\mathrm{supp}}

\newtheorem{theorem}{Theorem}[section]
\newtheorem{lemma}[theorem]{Lemma}
\newtheorem{proposition}[theorem]{Proposition}

\theoremstyle{definition}
\newtheorem{definition}[theorem]{Definition}

\theoremstyle{remark}

\newtheorem{remark}[theorem]{Remark}

\newtheorem{example}[theorem]{Example}

\numberwithin{equation}{section}

\title{Non-standard Green energy problems in the complex plane}

\author{Abey L\'{o}pez-Garc\'{i}a \qquad Alexander Tovbis}

\date{\today}

\begin{document}

\maketitle

\begin{abstract}
We consider several non-standard discrete and continuous Green energy problems in the complex plane and study the asymptotic relations between their solutions. In the discrete setting, we consider two problems; one with variable particle positions (within a given compact set) and variable particle masses, the other one with variable masses but prescribed positions. The mass of a particle is allowed to take any value in the range $0\leq m\leq R$, where $R>0$ is a fixed parameter in the problem. The corresponding continuous energy problems are defined on the space of positive measures $\mu$ with mass $\|\mu\|\leq R$ and supported on the given compact set, with an additional upper constraint that appears as a consequence of the prescribed positions condition. It is proved that the equilibrium constant and equilibrium measure vary continuously as functions of the parameter $R$ (the latter in the weak-star topology). In the unconstrained energy problem we present a greedy algorithm that converges to the equilibrium constant and equilibrium measure. In the discrete energy problems, it is shown that under certain conditions, the optimal values of the particle masses are uniquely determined by the optimal positions or prescribed positions of the particles, depending on the type of problem considered.

\smallskip


\smallskip

\end{abstract}

\section{Introduction}

Let $D$ be a domain (open connected set) in the extended complex plane $\overline{\mathbb{C}}=\mathbb{C}\cup\{\infty\}$ possessing a Green function $g(z,\zeta)$. Let $K\subset D\setminus\{\infty\}$ be a compact set and $f: K\rightarrow\mathbb{R}\cup\{+\infty\}$ be a lower semicontinuous function such that
\begin{equation}\label{admcond}
\mathrm{cap}(\{z\in K: f(z)<+\infty\})>0
\end{equation}
where $\mathrm{cap}(\cdot)$ denotes the logarithmic capacity. These assumptions hold throughout the paper. We will refer to $f$ as the \emph{external field}, and condition \eqref{admcond} is called \emph{admissibility} of $f$. Fix $R>0$, and let $\mathcal{B}_{R}(K)$ denote the space of all positive Borel measures supported on $K$ with total mass $\|\mu\|=\mu(K)\leq R$. For $\mu\in \mathcal{B}_{R}(K)$, let
\begin{equation}\label{def:energfunc}
J_{f}(\mu):=\iint g(z,\zeta)\,d\mu(z)\,d\mu(\zeta)+2\int f\, d\mu.
\end{equation}
The first term is the Green energy of $\mu$. Observe that the energy functional $J_{f}: \mathcal{B}_{R}(K)\rightarrow\mathbb{R}\cup\{+\infty\}$ is bounded below. This follows from the positivity of the Green function and the lower semicontinuity of $f$. Let
\begin{equation}\label{def:vf}
v_{f}:=\inf\{J_{f}(\mu): \mu\in \mathcal{B}_{R}(K)\}.
\end{equation}
By taking the zero measure $\mu=0$, we see that $J_{f}(0)=0$, therefore $v_{f}\leq 0$. The first problem we consider is the minimization of the functional $J_{f}$. The existence and uniqueness of a measure $\mu^{*}=\mu^{*}_{R}\in \mathcal{B}_{R}(K)$ satisfying $J_{f}(\mu^{*})=v_{f}$ is proved in Theorem~\ref{theo:fund}. 

The interesting case to consider is when $\min_{z\in K} f(z)<0$, since $\min_{z\in K} f(z)\geq 0$ trivially implies $\mu^{*}_{R}=0$. If $\min_{z\in K} f(z)<0$, as $R$ increases from zero the equilibrium measure $\mu_{R}^{*}$ may vary with $R$, but it eventually stabilizes to a constant measure (this can be deduced from formula \eqref{eq:reducspace}). This constant measure is therefore the minimizer of the energy functional $J_{f}$ among all positive and finite measures supported on $K$. This global minimizer is described in the so-called Gauss-Frostman theorem \cite[Theorem II.5.16]{SaffTotik}. In Theorem~\ref{theo:dynamic} it is proved that the equilibrium measure $\mu_{R}^{*}$ and equilibrium constant $v_{f}=v_{f}(R)$ depend continuously on $R$ (the measure in the weak-star topology).

In potential theory, the vast majority of the problems investigated deal with spaces of measures with fixed total mass, see for example the monographs \cite{BHS,Landkof,NikSor,Rans,SaffTotik}. In this sense, we refer to the variable-mass problems studied in this paper as non-standard energy problems.

Another problem considered is to minimize the functional $J_{f}$ over the space $\{\mu: 0\leq \mu\leq R\lambda\}$, where $\lambda$ is a probability measure on $K$ with finite Green energy (equivalently, finite logarithmic energy). For measures $\sigma$ and $\tau$ we write $\sigma\leq \tau$ if $\tau-\sigma$ is a positive measure. Let
\begin{equation}\label{def:vflambda}
v_{f,\lambda}:=\inf\{J_{f}(\mu): 0\leq \mu\leq R\lambda\}.
\end{equation}
Since $0\leq \mu\leq R\lambda$ implies $\mu\in \mathcal{B}_{R}(K)$, we have
\begin{equation}\label{ineqvfvfl}
v_{f}\leq v_{f,\lambda}\leq 0.
\end{equation}
The existence and uniqueness of a minimizer $\mu_{\lambda}^{*}=\mu_{\lambda,R}^{*}$ in this constrained problem is proved in Theorem~\ref{theo:fund:constr}. As in the previous unconstrained problem, $\mu_{\lambda}^{*}$ and $v_{f,\lambda}$ depend continuously on $R$, see Theorem~\ref{theo:dynamic:constr}.

We discuss now a pair of discrete energy problems that correspond to the continuous problems stated above, and describe the asymptotic relations between their solutions. 

Let $I:=[0,R]$. For each $N\geq 2$ we define the energy functional
\begin{align*}
E_{N,f}(z_{1},\ldots,z_{N},m_{1},\ldots,m_{N}) & :=\sum_{1\leq i\neq j\leq N}m_{i}\,m_{j}\, g(z_{i},z_{j}) +2N\sum_{j=1}^{N} m_{j}\, f(z_{j})\\
& =2\sum_{1\leq i<j\leq N}m_{i}\,m_{j}\,g(z_{i},z_{j}) +2N\sum_{j=1}^{N} m_{j}\, f(z_{j}) 
\end{align*}
for $(z_{1},\ldots,z_{N},m_{1},\ldots,m_{N})\in K^{N}\times I^{N}$. In this definition, we use the usual convention $0\cdot \infty=0$. Then $E_{N,f}:K^{N}\times I^{N}\longrightarrow\mathbb{R}\cup\{+\infty\}$ is well defined and it is trivially bounded below. Let $e_{N}=e_{N,f}$ denote the constant
\begin{equation}\label{def:eN}
e_{N}:=\inf_{K^{N}\times I^{N}} E_{N,f}(z_{1},\ldots,z_{N},m_{1},\ldots,m_{N}).
\end{equation}
By taking $m_{1}=\cdots=m_{N}=0$, we see that $e_{N}\leq 0$. Since $E_{N,f}$ is a lower semicontinuous function on the compact set $K^{N}\times I^{N}$, the constant \eqref{def:eN} is attained, i.e.
\begin{equation}\label{eq:eNattain}
e_{N}=E_{N,f}(z_{1,N},\ldots,z_{N,N}, m_{1,N},\ldots,m_{N,N})
\end{equation}
for some points $z_{j,N}\in K$, $m_{j,N}\in I$. 

We state our first main result:

\begin{theorem}\label{theo:main:1}
We have
\begin{equation}\label{eq:limeNvf}
\lim_{N\rightarrow\infty}\frac{e_{N}}{N^2}=v_{f}.
\end{equation}
For each $N\geq 2$, let $(z_{1,N},\ldots,z_{N,N},m_{1,N},\ldots,m_{N,N})$ be a configuration in $K^{N}\times I^{N}$ such that \eqref{eq:eNattain} holds, and let
\begin{equation}\label{def:muN}
\mu_{N}:=\frac{1}{N}\sum_{j=1}^{N}m_{j,N}\, \delta_{z_{j,N}}\in \mathcal{B}_{R}(K).
\end{equation}
Then the sequence $(\mu_{N})_{N=2}^{\infty}$ converges in the weak-star topology to $\mu^*$, the unique minimizer of $J_{f}$ on $\mathcal{B}_{R}(K)$.  
\end{theorem}

In Theorem \ref{theo:basicprop} we describe several basic properties of the minimizers of the energy functionals $J_{f}$ and $E_{N,f}$, see also Remark \ref{rem:lineareqns}. A natural question that arises is whether the separate probability distributions $\frac{1}{N}\sum_{j}\delta_{z_{j,N}}$ and $\frac{1}{N}\sum_{j}\delta_{m_{j,N}}$ have a limit under appropriate conditions on $f$, and if so, how to describe the limits. In the trivial case $f\geq 0$ on $K$, then clearly $e_{N}=0$ for all $N$, and so the values $m_{1,N},\ldots,m_{N,N}$ must all be zero except possibly one of them; therefore $\frac{1}{N}\sum_{j}\delta_{m_{j,N}}$ converges to $\delta_{0}$. But it is easy to see that in this case the measure $\frac{1}{N}\sum_{j}\delta_{z_{j,N}}$ is not necessarily convergent. It is not clear to us whether both probability distributions have a limit if $f<0$ on a sufficiently large subset of $K$. We leave this issue as an open question. 

In Section~\ref{sec:greedy} we describe a greedy algorithm in which a functional of only two variables (instead of $2N$ variables as in $E_{N,f}$) is minimized at each step, and generates an infinite sequence $((a_{N},m_{N}))_{N=1}^{\infty}$ with the property that the configurations $(a_{1},\ldots,a_{N},m_{1},\ldots,m_{N})$ behave asymptotically like the optimal configurations $(z_{1,N},\ldots,z_{N,N}, m_{1,N},\ldots,m_{N,N})$ analyzed in Theorem~\ref{theo:main:1}.

Finally, we discuss the discrete problem related to the continuous problem with upper constraint. Consider a sequence of point configurations 
\begin{equation}\label{triangarray}
(x_{j,N})_{1\leq j\leq l_{N}},\quad N\geq 1,\quad l_{N}\geq 2,
\end{equation}
in the compact set $K\subset D\setminus\{\infty\}$, where $l_{N}/N\rightarrow 1$ as $N\rightarrow\infty$, satisfying the following properties:
\begin{itemize}
\item[P1)] The sequence of discrete measures $\lambda_{N}:=\frac{1}{l_N}\sum_{j=1}^{l_N}\delta_{x_{j,N}}$ converges in the weak-star topology to a probability measure $\lambda$ on $K$.
\item[P2)] For every $N\geq 1$, there exists a partition $\mathcal{P}_N=\{V_{1,N},\ldots,V_{l_N,N}\}$ of $K$ (i.e., the sets $V_{j,N}$ are pairwise disjoint and their union is $K$) such that $x_{j,N}\in V_{j,N}$, $1\leq j\leq l_{N}$, and we have:
\begin{itemize}
\item[i)] $\lambda(V_{j,N})\leq 1/N$ for every $1\leq j\leq l_{N}$.
\item[ii)] There exists a sequence $(\kappa_{N})_{N=1}^{\infty}$ of positive numbers that converges to zero such that $\mathrm{diam}(V_{j,N})=\sup\{|x-y|: x,y\in V_{j,N}\}\leq \kappa_N$ for all $1\leq j\leq l_N$.
\end{itemize} 
\item[P3)] There exists an absolute constant $C>0$ such that $|x_{i,N}-x_{j,N}|\geq \frac{C}{N}$ for all $i\neq j$, $N\geq 1$.
\item[P4)] $\lambda$ satisfies $\iint g(z,\zeta)\,d\lambda(z)\,d\lambda(\zeta)<\infty$. 
\end{itemize}

We call $\lambda$ the \emph{reference} measure. A simple construction of configurations \eqref{triangarray} and partitions satisfying all the properties $\mathrm{P1)}$--$\mathrm{P4)}$ is given in Example \ref{examplemain} in Section~\ref{sec:constrained}. Under the assumptions $\mathrm{P1)}$--$\mathrm{P4)}$, we define the extremal constants
\begin{equation}\label{defdN}
d_{N}:=\inf\{E_{l_{N},f}(x_{1,N},\ldots,x_{l_{N},N},m_{1},\ldots,m_{l_{N}}): m_{1},\ldots,m_{l_{N}}\in I\}.
\end{equation}
Note that here we only minimize in the variables $(m_{1},\ldots,m_{l_{N}})\in I^{l_N}=[0,R]^{l_N}$ and the particle positions $x_{j,N}$ are fixed. It is obvious that we have
\begin{equation}\label{ineqeNdN}
e_{l_{N}}\leq d_{N}\leq 0,\qquad N\geq 1.
\end{equation}
Observe that if $f$ takes finite values at the points $x_{j,N}$, then in \eqref{defdN} we are actually minimizing the quadratic form plus linear form $m\mapsto m\,G_{N}\,m^{t}+F_{N}\,m^{t}$, $m=(m_{1},\ldots,m_{l_{N}})$, where $G_{N}$ is the symmetric matrix of size $l_{N}\times l_{N}$
\[
G_{N}=\begin{pmatrix}
0 & g(x_{1,N},x_{2,N}) & \cdots & g(x_{1,N},x_{l_N,N})\\
g(x_{2,N},x_{1,N}) & 0 & \cdots & g(x_{2,N},x_{l_N,N})\\
\vdots & \vdots & \ddots & \vdots \\
g(x_{l_{N},N},x_{1,N}) & g(x_{l_{N},N},x_{2,N}) & \cdots & 0
\end{pmatrix}
\]
and $F_{N}=2\,l_{N} (f(x_{1,N}),\ldots,f(x_{l_{N},N}))$.

Our final result describes, under certain conditions, the asymptotics of $d_{N}$ and the discrete measure $\frac{1}{l_{N}}\sum_{j=1}^{l_{N}}\widehat{m}_{j,N}\,\delta_{x_{j,N}}$, where $\widehat{m}_{j,N}$ are optimal constants for \eqref{defdN}. The limits are determined by $v_{f,\lambda}$ and $\mu_{\lambda}^{*}$, where $\lambda$ is the reference measure. In this paper the logarithmic potential of a measure $\mu$ is denoted
\[
U^{\mu}(z)=\int\log\frac{1}{|z-t|}\,d\mu(t),
\]
$\mathrm{Co}(A)$ denotes the convex hull of the set $A$, and $B(z,r)$ denotes the open disk with center $z\in\mathbb{C}$ and radius $r>0$. 

Before formulating the following Theorem \ref{theo:main:upcon}, we introduce some notions. First, let $\Gamma\subset D\setminus\{\infty\}$ be a curve that has a continuous and one-to-one parametrization $\gamma:[a,b]\rightarrow\Gamma$. We say that $\Gamma$ is locally monotonic if there exists a constant $\rho>0$ such that $B(x,\rho)\cap\Gamma$ is a connected set for  all $x\in\Gamma$ and, moreover, the distance $|y-x|$ is monotonically decreasing as a point $y\in B(x,\rho)\cap\Gamma$ traverses $\Gamma$ towards the point $x$. More specifically, for all $x\in\Gamma$ the set $I_{x}:=\gamma^{-1}(B(x,\rho)\cap\Gamma)\subset  [a,b]$  is an interval  and if $\gamma(s)=x$, then the function $r\mapsto|\gamma(r)-x|$ is non-increasing on the interval $[a,s]\cap I_{x}$ and it is non-decreasing on the interval $[s,b]\cap I_{x}$. We also say that a  partition $\mathcal P_N$ from  property $\mathrm{P2)}$ is monotonic if the sets $V_{j,N}\in \mathcal P_N$
 satisfy $\mathrm{Co}(\gamma^{-1}(V_{j,N}))\cap\mathrm{Co}(\gamma^{-1}(V_{i,N}))=\emptyset$ for all $i\neq j$. 

\begin{theorem}\label{theo:main:upcon}
Assume that $K\subset \Gamma\subset D\setminus\{\infty\}$, where $\Gamma$ is a locally monotonic curve that has a continuous and one-to-one parametrization $\gamma:[a,b]\rightarrow\Gamma$. Assume further that $f$ and $U^{\lambda}$ are continuous on $K$, and the partitions $\mathcal P_N$  in property $\mathrm{P2)}$  are monotonic for all $N$. Then
\begin{equation}\label{asympdNnew}
\lim_{N\rightarrow\infty}\frac{d_{N}}{N^2}=v_{f,\lambda}.
\end{equation}
Additionally, if $\{\widehat{m}_{j,N}\}_{j=1}^{l_{N}}\subset[0,R]$ is a collection of numbers such that 
\[
d_{N}=E_{l_{N},f}(x_{1,N},\ldots,x_{l_{N},N},\widehat{m}_{1,N},\ldots,\widehat{m}_{l_{N},N}),\qquad\mbox{for all}\,\,\,\,N\geq 1,
\]
then the sequence of measures $\widehat{\mu}_{N}=\frac{1}{l_{N}}\sum_{j=1}^{l_{N}}\widehat{m}_{j,N}\,\delta_{x_{j,N}}$ converges in the weak-star topology to $\mu_{\lambda}^{*}$.
\end{theorem} 

Note that for the extremal values $d_{N}$ in \eqref{defdN} to be defined, we only need the configurations \eqref{triangarray} to be given, and none of the properties $\mathrm{P1)}$--$\mathrm{P4)}$ are really necessary. However, it is natural to assume some regularity conditions in the problem setting in order to obtain some non-trivial asymptotic results. In the general framework of Theorem \ref{theo:main:upcon}, an interesting question is whether it is possible, given a sequence of point configurations \eqref{triangarray} satisfying properties $\mathrm{P1)}$, $\mathrm{P3)}$, $\mathrm{P4)}$, to construct partitions $\mathcal{P}_{N}$ satisfying property $\mathrm{P2)}$. It seems to us that this is not an easy problem. However, if one leaves aside properties $\mathrm{P1)}$, $\mathrm{P3)}$, $\mathrm{P4)}$, there is a standard procedure to construct monotonic partitions satisfying property $\mathrm{P2)}$, see Example \ref{examplemain1} in Section \ref{sec:constrained}.

We expect the asymptotic results in Theorem~\ref{theo:main:upcon} to be valid under weaker conditions, especially related to the geometry of the set $K$. This issue, and some applications of our results to certain problems in \cite{KuijTov,TovWang} involving Green potentials, will be addressed in a future work. As in Theorem~\ref{theo:main:1}, we leave here as an open question whether the probability measure $\frac{1}{l_{N}}\sum_{j=1}^{l_{N}}\delta_{\widehat{m}_{j,N}}$ converges under appropriate conditions on $f$. See Remark~\ref{rem:altconfig} for an alternative and slightly simpler construction of nodes $(x_{j,N})$ that yields the same asymptotic results in Theorem~\ref{theo:main:upcon}. 

Finally, we want to remark that energy problems with upper constraint have been very useful as a tool to solve important problems in approximation theory, numerical linear algebra, and other areas. We refer the interested reader to the seminal works of Rakhmanov \cite{Rakh}, Dragnev--Saff \cite{DragnevSaff}, and other authors  \cite{BeckGry,BeckKuij,KuijVan}. Problems with mass constraint of the form $\|\mu\|\leq R$ are more common in mathematical physics, see for example Lieb's survey \cite{Lieb}.

This paper is organized as follows. In Section~\ref{section:freelocmass} we discuss the unconstrained variational problem and prove Theorem~\ref{theo:main:1}. In Section~\ref{sec:greedy} we discuss the greedy approximation construction in the unconstrained case. In Section~\ref{sec:constrained} we analyze the constrained variational problem and prove Theorem~\ref{theo:main:upcon}.

\section{Variational problem in a ball, and discrete energy with free particle positions and masses}\label{section:freelocmass}

The Green potential of $\mu$ is denoted
\[
U^{\mu}_{G}(z)=\int g(z,\zeta)\,d\mu(\zeta).
\]
Let
\begin{equation}\label{def:betakappa}
\beta:=\min_{z\in K} f(z),\qquad \kappa:=\min_{K\times K} g(z,\zeta)>0.
\end{equation}
If $\beta\geq 0$, then trivially $v_{f}=0$. Assume now that $\beta<0$. Then for $\mu\in \mathcal{B}_{R}(K)$, we have
\[
J_{f}(\mu)\geq 2\int f\,d\mu\geq 2\beta\mu(K)\geq 2\beta R.
\]
Additionally, if $\mu$ is any finite positive measure on $K$ (not necessarily with mass $\leq R$), then
\begin{equation}\label{eq:ineqenerg}
J_{f}(\mu)\geq \kappa (\mu(K))^2+2\beta \mu(K)=\kappa\left(\mu(K)+\frac{\beta}{\kappa}\right)^2-\frac{\beta^2}{\kappa}\geq -\frac{\beta^2}{\kappa}.
\end{equation}
We deduce that if $\beta<0$, then
\[
\max\left\{2\beta R, -\frac{\beta^2}{\kappa}\right\}\leq v_{f}\leq 0.
\] 
Observe also that if $\beta<0$ and $\mu$ is a finite positive measure on $K$ with $\|\mu\|>-\frac{2\beta}{\kappa}$, then from the first inequality in \eqref{eq:ineqenerg} we get $J_{f}(\mu)>0\geq v_{f}$. In conclusion, in the case $\beta<0$ we have
\begin{equation}\label{eq:reducspace}
v_{f}=\inf\{J_{f}(\mu): \mu\in \mathcal{B}_{R^{*}}(K)\},\qquad R^{*}=\min\{R, -2\beta/\kappa\}.
\end{equation}

The following result is a bounded mass version of the Gauss-Frostman theorem (Theorem II.5.6 in \cite{SaffTotik}). In the Gauss-Frostman theorem the energy minimization takes place in the space of all positive and finite Borel measures supported on $K$. We follow closely the argument in the proof of that theorem. As it was already mentioned in the introduction, for $R>0$ sufficiently large, the minimizer of the energy functional $J_{f}$ in the ball $\mathcal{B}_{R}(K)$ coincides with the equilibrium measure $\tau_{0}$ in the Gauss-Frostman theorem, which satisfies the inequalities
\begin{align*}
U^{\tau_{0}}_{G}(z)+f(z) & \geq 0 \quad \mbox{q.e. on}\,\,K,\\
U^{\tau_{0}}_{G}(z)+f(z) & \leq 0 \quad \mbox{for all}\,\,z\in \mathrm{supp}(\tau_{0}).
\end{align*}
Results on existence and uniqueness of equilibrium measures with external fields in spaces of positive measures with fixed mass are numerous in the literature. A good compendium of such results can be found in \cite{BHS,SaffTotik}. For recent developments in this theory, in addition to the references already cited, see \cite{BCC,BrauDragSaff,DragOriSaffWiel,KuijTov,Lop,Zorii1,Zorii2,Zorii3}. 

\begin{theorem}\label{theo:fund}
Let $R>0$. For an admissible external field $f$, there exists a unique measure $\mu^*\in \mathcal{B}_{R}(K)$ that minimizes the energy functional \eqref{def:energfunc}, i.e., satisfying $J_{f}(\mu^*)=v_{f}$. The following inequalities hold: 
\begin{align}
U^{\mu^{*}}_{G}(z)+f(z) & \geq C_{f} \quad \mbox{q.e. on}\,\,K,\label{varineq1}\\
U^{\mu^{*}}_{G}(z)+f(z) & \leq C_{f} \quad \mbox{for all}\,\,z\in \mathrm{supp}(\mu^{*}),\label{varineq2}
\end{align}
where 
\begin{equation}\label{def:Cf}
C_{f}=\frac{1}{R}\left(\iint g(z,\zeta)\,d\mu^{*}(z)\,d\mu^{*}(\zeta)+\int f\,d\mu^{*}\right).
\end{equation}
We have $C_{f}\leq 0$, and $C_{f}=0$ if $\|\mu^{*}\|<R$. The support of $\mu^{*}$ is contained in the set $\{z\in K: f(z)\leq C_{f}-\kappa\,\|\mu^{*}\|\}$, where $\kappa$ is defined in \eqref{def:betakappa}.
\end{theorem}
\begin{proof}
We justify first the uniqueness of the minimizer. Suppose that $\mu_{1}, \mu_{2}\in \mathcal{B}_{R}(K)$ satisfy $J_{f}(\mu_{1})=J_{f}(\mu_{2})=v_{f}$. Since $v_{f}$ is finite, it follows that $\mu_1$ and $\mu_2$ have finite Green energy. Therefore, for the signed measure $\nu:=\frac{1}{2}(\mu_{1}-\mu_{2})$ we know by Theorem II.5.6 in \cite{SaffTotik} that 
\[
\mathcal{I}:=\iint g(z,\zeta)\,d\nu(z)\,d\nu(\zeta)\geq 0
\]
with equality if and only if $\nu=0$. Now, we have $\frac{1}{2}(\mu_{1}+\mu_{2})\in \mathcal{B}_{R}(K)$, so $J_{f}\left(\frac{1}{2}(\mu_{1}+\mu_{2})\right)\geq v_{f}$. A simple computation shows that
\[
\mathcal{I}+J_{f}\left(\frac{1}{2}(\mu_{1}+\mu_{2})\right)=\frac{1}{2}J_{f}(\mu_{1})+\frac{1}{2}J_{f}(\mu_{2})=v_{f},
\]
hence $\mathcal{I}\leq 0$. Therefore $\mathcal{I}=0$ and $\mu_{1}=\mu_{2}$.

To prove the existence of a minimizer, let $(\nu_{N})$ be a sequence of measures in $\mathcal{B}_{R}(K)$ such that
\[
J_{f}(\nu_{N})<v_{f}+\frac{1}{N},\qquad N\geq 1.
\]
By Helly's selection theorem, there exists a subsequence $(\nu_{N_{k}})$ that converges in the weak-star topology to a measure $\nu\in \mathcal{B}_{R}(K)$. By the lower semicontinuity of $g$ and $f$, we get
\[
J_{f}(\nu)\leq \liminf_{k\rightarrow\infty}J_{f}(\nu_{N_{k}})=v_{f}
\] 
hence $J_{f}(\nu)=v_{f}$.

Assume now that $\|\mu^{*}\|<R$ (this is the case if, for example, $\beta<0$ and $R>-2\beta/\kappa$, see \eqref{eq:reducspace}) and let us show that in this case
\begin{equation}\label{eq:id1}
\iint g(z,\zeta)\,d\mu^{*}(z)\,d\mu^{*}(\zeta)+\int f\,d\mu^{*}=0.
\end{equation}
Let
\[
r:=\frac{R-\|\mu^{*}\|}{R}>0.
\]
If $|\epsilon|<r$, then $(1+\epsilon)\mu^{*}\in \mathcal{B}_{R}(K)$. Therefore we have $J_{f}((1+\epsilon)\mu^{*})\geq J_{f}(\mu^{*})$, which reduces to 
\[
\epsilon^{2}\iint g(z,\zeta)\,d\mu^{*}(z)\,d\mu^{*}(\zeta)+2\epsilon\left(\iint g(z,\zeta)\,d\mu^{*}(z)\,d\mu^{*}(\zeta)+\int f\,d\mu^{*}\right)\geq 0.
\]
This inequality is valid for all $\epsilon\in(-r,r)$, so \eqref{eq:id1} follows. If $\|\mu^{*}\|=R$, then $(1-\epsilon)\mu^{*}\in \mathcal{B}_{R}(K)$ for all $0<\epsilon<1$, and the inequality $J_{f}((1-\epsilon)\mu^{*})\geq J_{f}(\mu^{*})$ implies that $C_{f}\leq 0$.

Assume now that $0<\|\mu^{*}\|\leq R$. Let $\mathcal{P}(K)$ denote the space of all Borel probability measures on $K$ and let
\[
\widehat{v}_{f}:=\{J_{Q}(\mu): \mu\in\mathcal{P}(K)\},
\]
where $Q=f/\|\mu^{*}\|$. It is easy to check that the probability measure $\mu^{*}/\|\mu^{*}\|$ is the minimizer of $J_{Q}$ in the space $\mathcal{P}(K)$, therefore by Theorem II.5.10 in \cite{SaffTotik}, we know that
\begin{align*}
U^{\mu^{*}}_{G}(z)+f(z) & \geq \|\mu^{*}\|\left(\widehat{v}_{f}-\frac{1}{\|\mu^{*}\|^2}\int f\,d\mu^{*}\right)\quad \mbox{q.e. on}\,\,K,\\
U^{\mu^{*}}_{G}(z)+f(z) & \leq \|\mu^{*}\|\left(\widehat{v}_{f}-\frac{1}{\|\mu^{*}\|^2}\int f\,d\mu^{*}\right)\quad \mbox{for all}\,\,z\in\mathrm{supp}(\mu^{*}).
\end{align*}
Since $\widehat{v}_{f}=J_{Q}(\mu^{*}/\|\mu^{*}\|)$, we obtain
\begin{align*}
\widehat{v}_{f}-\frac{1}{\|\mu^{*}\|^2}\int f\,d\mu^{*} & =\frac{1}{\|\mu^{*}\|^2}\iint g(z,\zeta)\,d\mu^*(z)\,d\mu^{*}(\zeta)+\frac{2}{\|\mu^{*}\|^2}\int f\,d\mu^{*}-\frac{1}{\|\mu^{*}\|^2}\int f\,d\mu^{*}\\
& =\frac{1}{\|\mu^{*}\|^{2}}\left(\iint g(z,\zeta)\,d\mu^{*}(z)\,d\mu^{*}(\zeta)+\int f\,d\mu^{*}\right).
\end{align*}
In virtue of \eqref{eq:id1}, this constant is zero if $0<\|\mu^{*}\|<R$, and it equals $C_{f}/R$ if $\|\mu^{*}\|=R$. This finishes the proof of \eqref{varineq1}--\eqref{varineq2} in the case $\|\mu^{*}\|>0$.

Assume that $\mu^{*}=0$. In this case we only need to justify \eqref{varineq1}, which reduces to show that $f\geq 0$ q.e. on $K$. If $\mu$ is any measure in $\mathcal{B}_{R}(K)$ with finite Green energy, then for any $0<\delta<1$, we have
\[
J_{f}(\delta\mu)=\delta^{2}\iint g(z,\zeta)\,d\mu(z)\,d\mu(\zeta)+2\delta \int f\,d\mu\geq v_{f}=0,
\]
which implies that 
\begin{equation}\label{posint}
\int f\,d\mu\geq 0.
\end{equation}
If $\{z\in K: f(z)<0\}$ has positive capacity, then for some positive integer $n$, the set $E_{n}:=\{z\in K: f(z)\leq -1/n\}$ also has positive capacity (see Theorem 5.1.3 in \cite{Rans}). Let $\mu_{E_{n}}$ be the Green equilibrium measure for the compact set $E_{n}$, i.e., the probability measure on $E_{n}$ with minimal Green energy. Then the Green energy of the measure $\mu=\frac{R}{2}\mu_{E_{n}}\in \mathcal{B}_{R}(K)$ is finite and we have $\int f\,d\mu<0$, which contradicts \eqref{posint}. This concludes the proof of \eqref{varineq1}--\eqref{varineq2}. 

The fact that $\supp(\mu^{*})\subset\{z\in K: f(z)\leq C_{f}-\kappa\,\|\mu^{*}\|\}$ follows immediately from \eqref{varineq2}.
\end{proof}

In the following result we gather some basic properties of the minimizers of the energy functionals $J_{f}$ and $E_{N,f}$.

\begin{theorem}\label{theo:basicprop}
Let $R>0$, and for an admissible external field $f$, let $\mu^{*}\in \mathcal{B}_{R}(K)$ be the minimizer of the energy functional $J_{f}$ in the ball $\mathcal{B}_{R}(K)$, and let $\beta=\min_{z\in K}f(z)$. The following properties hold:
\begin{itemize}
\item[$i)$] $\mu^{*}=0$ is the zero measure if and only if $f\geq 0$ q.e. on $K$.
\item[$ii)$] If $\|\mu^{*}\|<R$, then $v_{f}=\int f\,d\mu^{*}=-\iint g(z,\zeta)\,d\mu^{*}(z)\,d\mu^{*}(\zeta)$.
\item[$iii)$] If $\|\mu^{*}\|>0$, then $v_{f}<0$.
\item[$iv)$] If $f\leq 0$ on $K$, then we have $e_{N+1}\leq e_{N}$ for all $N\geq 2$.
\item[$v)$] Let $(z_{1,N},\ldots,z_{N,N},m_{1,N},\ldots,m_{N,N})\in K^{N}\times I^{N}$ be a minimizer of $E_{N,f}$ for some $N\geq 2$. If $m_{i,N}>0$ and $m_{j,N}>0$ for a pair of indices $i\neq j$, then $z_{i,N}\neq z_{j,N}$. 
\item[$vi)$] If $\beta<0$ and $e_{N+1}<e_{N}+2RN\beta$, then for any minimizer $(z_{1,N+1},\ldots,m_{N+1,N+1})$ of $E_{N+1,f}$ we must have $m_{j,N+1}>0$ for all $1\leq j\leq N+1$.
\item[$vii)$] If $(z_{1,N},\ldots,z_{N,N}, m_{1,N},\ldots,m_{N,N})$ is a minimizer of $E_{N,f}$ and $0<m_{i,N}<R$ for some $1\leq i\leq N$, then $f(z_{i,N})$ is finite and we have
\[
\sum_{\substack{j=1\\ j\neq i}}^{N} m_{j,N}\,g(z_{i,N},z_{j,N})=-N f(z_{i,N}).
\]
\end{itemize}
\end{theorem}
\begin{proof}
If $\mu^{*}=0$, then by \eqref{varineq1} we have $f\geq C_{f}=0$ q.e. on $K$. Conversely, suppose that $f\geq 0$ q.e. on $K$. Since $\mu^{*}$ has finite logarithmic energy, by Theorem 3.2.3 in \cite{Rans} we deduce that $f\geq 0$ $\mu^{*}$-a.e., and so $v_{f}\geq 0$. Therefore $v_{f}=0$ and $\mu^{*}=0$, so $i)$ is justified. 

If $\|\mu^{*}\|<R$, then by Theorem~\ref{theo:fund} we have $C_{f}=0$, hence $v_{f}=J_{f}(\mu^{*})=\int f\,d\mu^{*}=-\iint g(z,\zeta)\,d\mu^{*}(z)\,d\mu^{*}(\zeta)$.

If $\|\mu^{*}\|>0$ and $v_{f}=0$, then the zero measure and $\mu^{*}$ are two different minimizers of $J_{f}$, which is impossible. Therefore $\|\mu^{*}\|>0$ implies $v_{f}<0$.

Now we justify $iv)$, so assume that $f\leq 0$ on $K$. Let $(z_{1,N},\ldots,z_{N,N},m_{1,N},\ldots,m_{N,N})\in K^{N}\times I^{N}$ be a minimizer of $E_{N,f}$, let $\widetilde{z}$ be an arbitrary point in $K$, and let $\widetilde{m}=0$. Then
\begin{gather*}
e_{N}=2\sum_{1\leq i<j\leq N} m_{i,N}\,m_{j,N}\,g(z_{i,N},z_{j,N})+2(N+1)\sum_{j=1}^{N} m_{j,N}f(z_{j,N})-2\sum_{j=1}^{N}m_{j,N} f(z_{j,N})\\
\geq 2\sum_{1\leq i<j\leq N} m_{i,N}\,m_{j,N}\,g(z_{i,N},z_{j,N})+2\sum_{i=1}^{N}m_{i,N}\,\widetilde{m}\,g(z_{i,N},\widetilde{z})\\+2(N+1)\left(\sum_{j=1}^{N} m_{j,N}f(z_{j,N})+\widetilde{m} f(\widetilde{z})\right)\geq e_{N+1}
\end{gather*}
where the last inequality follows by definition of $e_{N+1}$.

If $(z_{1,N},\ldots,z_{N,N}, m_{1,N},\ldots,m_{N,N})$ is a minimizer of $E_{N,f}$ and $m_{i,N}, m_{j,N}>0$ for indices $i\neq j$, then $z_{i,N}\neq z_{j,N}$, otherwise $m_{i,N}\,m_{j,N}\,g(z_{i,N},z_{j,N})=m_{i,N}\,m_{j,N}\,g(z_{i,N},z_{i,N})=\infty$, which is impossible since $e_{N}$ is finite.

Now we justify $vi)$. Suppose that $\beta<0$ and $e_{N+1}<e_{N}+2RN\beta$, and take an arbitrary minimizer $(z_{1,N+1},\ldots,z_{N+1,N+1},m_{1,N+1},\ldots,m_{N+1,N+1})$ of $E_{N+1,f}$. Suppose that $m_{j_{0},N+1}=0$ for some $1\leq j_{0}\leq N+1$, and without loss of generality assume that $j_{0}=N+1$. Then
\begin{gather*}
e_{N+1}=2\sum_{1\leq i<j\leq N+1} m_{i,N+1}\,m_{j,N+1}\,g(z_{i,N+1},z_{j,N+1})+2(N+1)\sum_{j=1}^{N+1} m_{j,N+1}f(z_{j,N+1})\\
=2\sum_{1\leq i<j\leq N} m_{i,N+1}\,m_{j,N+1}\,g(z_{i,N+1},z_{j,N+1})+2N\sum_{j=1}^{N} m_{j,N+1}f(z_{j,N+1})+2\sum_{j=1}^{N}m_{j,N+1} f(z_{j,N+1})\\
\geq e_{N}+2\sum_{j=1}^{N}m_{j,N+1} f(z_{j,N+1})\geq e_{N}+2RN\beta,
\end{gather*}
hence $e_{N+1}\geq e_{N}+2RN\beta$, which is a contradiction. 

Finally, we justify $vii)$, so assume that $(z_{1,N},\ldots,m_{N,N})$ is a minimizer of $E_{N,f}$ and $0<m_{i,N}<R$ for some $1\leq i\leq N$. If $f(z_{i,N})=+\infty$, then $e_{N}=+\infty$, which is impossible, so $f(z_{i,N})$ is finite. Consider the function
\begin{align*}
h(m) & :=E_{N,f}(z_{1,N},\ldots,z_{N,N},m_{1,N},\ldots,m_{i-1,N},m,m_{i+1,N},\ldots,m_{N,N})\\
& =2m\sum_{\substack{j=1\\ j\neq i}}^{N}m_{j,N}\,g(z_{i,N},z_{j,N})+2Nm f(z_{i,N})+C
\end{align*}
where $C$ is a constant, and observe that $m_{j,N}\,g(z_{i,N},z_{j,N})$ is finite for all $j\neq i$. This function $h$ defined on the interval $[0,R]$ is differentiable and has a minimum at the point $m_{i,N}\in(0,R)$, hence
\[
0=h'(m_{i,N})=2\sum_{\substack{j=1\\ j\neq i}}^{N}m_{j,N}\,g(z_{i,N},z_{j,N})+2N f(z_{i,N})
\]
so $vii)$ is justified.\end{proof}

\begin{remark}\label{rem:lineareqns}
It follows from Theorem \ref{theo:basicprop}, part $vii)$, that if $(z_{1,N},\ldots,z_{N,N},m_{1,N},\ldots,m_{N,N})$ is a minimizer of $E_{N,f}$ and $0<m_{i,N}<R$ for all $1\leq i\leq N$, then the column vector $\mathbf{v}:=(m_{1,N}\,\,\,m_{2,N}\,\,\,\ldots\,\,m_{N,N})^{t}$ satisfies the relation
\[
\widetilde{G}_{N}\,\mathbf{v}=\widetilde{F}_{N}
\] 
where $\widetilde{G}_{N}$ is the $N\times N$ symmetric matrix with entries 
\[
\widetilde{G}_{N}(i,j)=\begin{cases}
g(z_{i,N},z_{j,N}), & i\neq j,\\
0, & i=j,
\end{cases}
\]
and $\widetilde{F}_{N}$ is the column vector $\widetilde{F}_{N}=-N (f(z_{1,N})\,\,\,f(z_{2,N})\,\,\,\ldots\,\,\,f(z_{N,N}))^{t}$. So if $\widetilde{G}_{N}$ is invertible, the masses of the particles are uniquely determined by the positions of the particles and we have $\textbf{v}=\widetilde{G}_{N}^{-1} \widetilde{F}_{N}$.
\end{remark}

As functions of $R$, the optimal value \eqref{def:vf} and the minimizer of the energy functional \eqref{def:energfunc} are continuous, as the following result shows.

\begin{theorem}\label{theo:dynamic}
Let $\mu_{R}^{*}$ denote the minimizer of the energy functional \eqref{def:energfunc} in the ball $\mathcal{B}_{R}(K)$, $R\geq 0$, and let $v_{f}(R)=J_{f}(\mu_{R}^{*})$. If $(R_{N})_{N}$ is a convergent sequence of positive numbers and $R=\lim_{N\rightarrow\infty} R_{N}$, then $\mu_{R_{N}}^{*}$ converges in the weak-star topology to $\mu_{R}^{*}$ and $v_{f}(R)=\lim_{N\rightarrow\infty} v_{f}(R_{N})$.
\end{theorem}
\begin{proof}
By the compactness of the balls $\mathcal{B}_{r}(K)$ in the weak-star topology, the convergence of $\mu_{R_{N}}^{*}$ to $\mu_{R}^{*}$ will be justified if we show that any weak-star convergent subsequence of $(\mu^{*}_{R_{N}})$ converges to $\mu_{R}^{*}$. Let $(\mu^{*}_{R_{N_{k}}})$ be such a subsequence with limit $\tau$. We have $\|\tau\|=\lim_{k\rightarrow\infty}\|\mu_{R_{N_{k}}}^{*}\|\leq \lim_{k\rightarrow\infty} R_{N_{k}}=R$, so $\tau\in \mathcal{B}_{R}(K)$ and therefore $J_{f}(\mu_{R}^{*})\leq J_{f}(\tau)$. The energy functional $J_{f}$ is lower semicontinuous in the weak-star topology, so $J_{f}(\tau)\leq \liminf_{k\rightarrow\infty}J_{f}(\mu_{R_{N_{k}}}^{*})$. Let $\widehat{\mu}_k:=(R_{N_{k}}/R)\,\mu_{R}^{*}$. Then $\|\widehat{\mu}_{k}\|\leq R_{N_{k}}$, hence $J_{f}(\mu_{R_{N_{k}}}^{*})\leq J_{f}(\widehat{\mu}_{k})$. It is obvious that $\lim_{k\rightarrow\infty}J_{f}(\widehat{\mu}_{k})=J_{f}(\mu_{R}^{*})$, so we conclude that
\[
J_{f}(\mu_{R}^{*})\leq J_{f}(\tau)\leq \liminf_{k\rightarrow\infty}J_{f}(\mu_{R_{N_{k}}}^{*})\leq \limsup_{k\rightarrow\infty}J_{f}(\mu_{R_{N_{k}}}^{*})\leq J_{f}(\mu^{*}_{R}),
\]
therefore $J_{f}(\tau)=J_{f}(\mu^{*}_{R})$ and $\tau=\mu_{R}^{*}$. Taking as subsequence the sequence $(R_{N})$ itself, we have shown that $v_{f}(R_{N})=J_{f}(\mu_{R_{N}}^{*})$ converges to $v_{f}(R)=J_{f}(\mu_{R}^{*})$, so the second claim is proved. 
\end{proof}

\noindent\textbf{Proof of Theorem~\ref{theo:main:1}:} First we show that for any $N\geq 2$ and $\mu\in \mathcal{B}_{R}(K)$, we have
\begin{equation}\label{ineqeJ}
\frac{e_{N}}{N^2}\leq J_{f}(\mu).
\end{equation}
Since $e_{N}\leq 0$, this inequality trivially holds if $\mu=0$ is the zero measure. Suppose that $\mu$ has positive mass, $\|\mu\|>0$. By definition,
\[
e_{N}\leq 2\sum_{1\leq i<j\leq N} m_{i}\,m_{j}\,g(z_{i},z_{j})+2N\sum_{j=1}^{N} m_{j}\,f(z_{j})
\]
for all $(z_{1},\ldots,z_{N},m_1,\ldots,m_{N})\in K^{N}\times I^{N}$. We take in this inequality $m_{j}=\|\mu\|$ for all $1\leq j\leq N$, and we get
\begin{equation}\label{ineq:eN1}
e_{N}\leq 2 \|\mu\|^2\sum_{1\leq i<j\leq N} g(z_{i},z_{j})+2N\|\mu\|\sum_{j=1}^{N}f(z_{j})
\end{equation}
for all $(z_{1},\ldots,z_{N})\in K^{N}$. Let $\sigma=\mu/\|\mu\|$. Integrating both sides of \eqref{ineq:eN1} with respect to the probability measure $d\sigma(z_{1})\cdots d\sigma(z_{N})$, we get
\begin{align*}
e_{N} & \leq 2\|\mu\|^2\frac{N(N-1)}{2}\iint g(z,\zeta)\,d\sigma(z)\,d\sigma(\zeta)+2\|\mu\|N^2 \int f\,d\sigma \\
& =N(N-1)\iint g(z,\zeta)\,d\mu(z)\,d\mu(\zeta)+2N^2\int f\,d\mu
\end{align*}
and \eqref{ineqeJ} follows. 

From \eqref{ineqeJ} and \eqref{def:vf} we deduce that
\begin{equation}\label{ineq:eNvf}
\frac{e_{N}}{N^2}\leq v_{f}\qquad \mbox{for all}\,\,N\geq 2.
\end{equation}
Since $(e_{N}/N^2)_{N=2}^{\infty}$ is a bounded sequence, in order to prove \eqref{eq:limeNvf} it suffices to show that any convergent subsequence of $(e_{N}/N^2)_{N=2}^{\infty}$ converges to $v_{f}$. So let $(e_{N_{k}}/N_{k}^2)_{k=1}^{\infty}$ be a convergent subsequence. It follows from \eqref{ineq:eNvf} that
\begin{equation}\label{limsubub}
\lim_{k\rightarrow\infty}\frac{e_{N_{k}}}{N_{k}^2}\leq v_{f}.
\end{equation}
Consider the sequence of measures $(\mu_{N_{k}})_{k=1}^{\infty}$, where $\mu_{N}$ is given by \eqref{def:muN}. By Helly's selection theorem, there is a subsequence of $(\mu_{N_{k}})_{k=1}^{\infty}$ that converges in the weak-star topology to a measure in $\mathcal{B}_{R}(K)$. We take a subsequence of $(\mu_{N_{k}})_{k=1}^{\infty}$ with this property, which for simplicity of notation will be denoted again as $(\mu_{N_{k}})_{k=1}^{\infty}$. Let $\tau$ be the limit measure of $(\mu_{N_{k}})_{k=1}^{\infty}$. 

By the lower semicontinuity of $f$, we have
\begin{equation}\label{eq:liminf1}
\int f\,d\tau\leq \liminf_{k\rightarrow\infty} \int f\,d\mu_{N_{k}}.
\end{equation}
For a constant $M>0$, we define the truncated Green kernel
\begin{equation}\label{deftruncGK}
g_{M}(z,\zeta):=\begin{cases}
g(z,\zeta) & \mbox{if}\,\,g(z,\zeta)\leq M\\
M & \mbox{if}\,\,g(z,\zeta)>M.
\end{cases}
\end{equation}
Then $g_{M}(z,\zeta)$ is continuous on $K\times K$, hence
\begin{equation}\label{limgM}
\lim_{k\rightarrow\infty}\iint g_{M}(z,\zeta)\,d\mu_{N_{k}}(z)\,d\mu_{N_{k}}(\zeta)=\iint g_{M}(z,\zeta)\,d\tau(z)\,d\tau(\zeta).
\end{equation}
So for each fixed $M>0$, we deduce from \eqref{eq:liminf1} and \eqref{limgM} that 
\begin{equation}\label{eq:liminf2}
\iint g_{M}(z,\zeta)\,d\tau(z)\,d\tau(\zeta)+2\int f\,d\tau\leq \liminf_{k\rightarrow\infty}\left(\iint g_{M}(z,\zeta)\,d\mu_{N_{k}}(z)\,d\mu_{N_{k}}(\zeta)+2\int f\,d\mu_{N_{k}}\right).
\end{equation}
In what follows we will abbreviate and write $z_{j}=z_{j,N_{k}}$, $m_{j}=m_{j,N_{k}}$. We have
\begin{align}
\iint g_{M}(z,\zeta)\,d\mu_{N_{k}}(z)\,d\mu_{N_{k}}(\zeta) & =\sum_{1\leq i,j\leq N_{k}}\frac{m_{i}\,m_{j}}{N_{k}^2}\,g_{M}(z_{i},z_{j})\notag\\
& =\sum_{j=1}^{N_{k}}\frac{m_{j}^{2}}{N_{k}^2}\,g_{M}(z_{j},z_{j})+\sum_{1\leq i\neq j\leq N_{k}} \frac{m_{i}\,m_{j}}{N_{k}^2}\,g_{M}(z_{i},z_{j})\notag\\
& =M \sum_{j=1}^{N_{k}} \frac{m_{j}^{2}}{N_{k}^2}+\sum_{1\leq i\neq j\leq N_{k}} \frac{m_{i}\,m_{j}}{N_{k}^2}\,g_{M}(z_{i},z_{j})\notag\\
& \leq \frac{M R^2}{N_{k}}+\sum_{1\leq i\neq j\leq N_{k}} \frac{m_{i}\,m_{j}}{N_{k}^2}\,g(z_{i},z_{j})\label{estimate}
\end{align}
where in the last inequality we used that $g_{M}\leq g$ and $m_{j}\leq R$ for all $1\leq j\leq N_{k}$.

So by \eqref{eq:liminf2} and \eqref{estimate} we get
\begin{align}
\iint g_{M}(z,\zeta) d\tau(z) d\tau(\zeta)+2\int f d\tau & \leq \liminf_{k\rightarrow\infty}\Big(\sum_{i\neq j} \frac{m_{i}\,m_{j}}{N_{k}^2}g(z_{i},z_{j})+\frac{2}{N_{k}}\sum_{j=1}^{N_{k}}m_{j}\, f(z_{j})+\mathcal{O}(N_{k}^{-1})\Big)\notag\\
& =\lim_{k\rightarrow\infty} \frac{e_{N_{k}}}{N_{k}^2}.\label{halfestimate}
\end{align}
This inequality is valid for every $M>0$, so letting $M\rightarrow\infty$, by the Monotone Convergence Theorem we obtain
\[
v_{f}\leq J_{f}(\tau)=\lim_{M\rightarrow\infty}\left(\iint g_{M}(z,\zeta)\, d\tau(z)\, d\tau(\zeta)+2\int f\, d\tau\right)\leq \lim_{k\rightarrow\infty} \frac{e_{N_{k}}}{N_{k}^2}
\]
which together with \eqref{limsubub} implies 
\begin{equation}\label{finidentity}
\lim_{k\rightarrow\infty} \frac{e_{N_{k}}}{N_{k}^2}=v_{f}=J_{f}(\tau).
\end{equation}
This concludes the proof of \eqref{eq:limeNvf}. We also deduce from \eqref{finidentity} that $\tau$ is a minimizer of the energy functional $J_{f}$, hence by the uniqueness of such measure we obtain $\tau=\mu^*$.

With \eqref{eq:limeNvf} justified, it is easy to see that the above argument also shows the following. If $(\mu_{N_{k}})$ is any subsequence of $(\mu_{N})$ that converges in the weak-star topology to a measure $\tau\in \mathcal{B}_{R}(K)$, then $\tau=\mu^*$. Therefore, the sequence $(\mu_{N})$ converges to $\mu^*$ in the weak-star topology.\qed

\section{Greedy approximation}\label{sec:greedy}

In this section we describe a greedy algorithm to obtain asymptotically the equilibrium constant $v_{f}$ and the equilibrium measure $\mu^{*}$ minimizing the energy functional $J_{f}$ on $\mathcal{B}_{R}(K)$. 

\begin{definition}
Assume that $\|\mu^{*}\|>0$, and let
\[
A:=\{z\in K:\,U^{\mu^{*}}_{G}(z)+f(z)\leq C_{f}\},
\]
where $C_{f}$ is defined in \eqref{def:Cf}. By \eqref{varineq2}, the set $A$ is a non-empty compact subset of $K$ that contains the support of $\mu^{*}$. We say that a sequence $((a_{N},m_{N}))_{N=1}^{\infty}$ in $A\times I$ is a greedy $f$-energy sequence if for every $N\geq 2$,
\begin{equation}\label{minpropUNf}
U_{N,f}(a_{N},m_{N})=\inf_{(z,m)\in A\times I}U_{N,f}(z,m),
\end{equation}
where
\[
U_{N,f}(z,m):=m\sum_{k=1}^{N-1}(m_{k}\,g(z,a_{k})+f(z)),\qquad N\geq 2.
\]
\end{definition}

Since $U_{N,f}$ is lower semicontinuous, it is clear that a greedy $f$-energy sequence can be constructed recursively starting from an arbitrary $(a_{1},m_{1})\in A\times I$. 

\begin{theorem}\label{theo:greedy}
Assume that the equilibrium measure $\mu^{*}$ satisfies $0<\|\mu^{*}\|<R$. If $((a_{N},m_{N}))_{N=1}^{\infty}$ is a greedy $f$-energy sequence in $A\times I$, then
\begin{equation}\label{asympenergygreedy}
\lim_{N\rightarrow\infty}\frac{E_{N,f}(a_{1},\ldots,a_{N},m_{1},\ldots,m_{N})}{N^2}=v_{f}=\int f\,d\mu^{*},
\end{equation}
and the sequence of discrete measures $\frac{1}{N}\sum_{k=1}^{N}m_{k}\,\delta_{a_{k}}$ converges in the weak-star topology to $\mu^{*}$. We also have
\begin{equation}\label{asympfs}
\lim_{N\rightarrow\infty}\frac{1}{N}\sum_{k=1}^{N}m_{k}\,f(a_{k})=\int f\,d\mu^{*}
\end{equation}
and
\begin{equation}\label{asympgas}
\lim_{N\rightarrow\infty}\frac{\sum_{1\leq i\neq j\leq N} m_{i}\,m_{j}\,g(a_{i},a_{j})}{N^{2}}=-v_{f}=\iint g(z,\zeta)\,d\mu^{*}(z)\,d\mu^{*}(\zeta).
\end{equation}
The value of $m_{N}$ can be chosen to be either $0$ or $R$ for every $N\geq 2$.
\end{theorem}
\begin{proof}
For every $N\geq 2$ we have 
\begin{gather*}
E_{N,f}(a_{1},\ldots,a_{N},m_{1},\ldots,m_{N})=2\sum_{1\leq i<j\leq N} m_{i}\,m_{j}\,g(a_{i},a_{j})+2N\sum_{j=1}^{N} m_{j}\,f(a_{j})\\
=2\sum_{j=2}^{N}\Big(\sum_{i=1}^{j-1} m_{i}\,m_{j}\,g(a_{i},a_{j})+(j-1)\,m_{j}\, f(a_{j})+\sum_{i=1}^{j-1}m_{i}\,f(a_{i})\Big)+2\sum_{k=1}^{N}m_{k}\,f(a_{k})\\
=2\sum_{j=2}^{N}\Big(U_{j,f}(a_{j},m_{j})+\sum_{i=1}^{j-1}m_{i}\,f(a_{i})\Big)+2\sum_{k=1}^{N}m_{k}\,f(a_{k}).
\end{gather*}
Applying now \eqref{minpropUNf}, for every $(z,m)\in A\times I$ we get
\begin{gather*}
E_{N,f}(a_{1},\ldots,a_{N},m_{1},\ldots,m_{N})\leq 2\sum_{j=2}^{N}\Big(U_{j,f}(z,m)+\sum_{i=1}^{j-1}m_{i}\,f(a_{i})\Big)+2\sum_{k=1}^{N}m_{k}\,f(a_{k})\\
=2\sum_{j=2}^{N}\sum_{i=1}^{j-1}(m_{i}\,m\,g(z,a_{i})+m\,f(z)+m_{i}\,f(a_{i}))+2\sum_{k=1}^{N}m_{k}\,f(a_{k}).
\end{gather*}
Let $\sigma=\mu^{*}/\|\mu^{*}\|$. Taking $m=\|\mu^{*}\|$ and integrating both sides of the previous inequality with respect to the probability measure $\sigma$, we get
\[
E_{N,f}(a_{1},\ldots,a_{N},m_{1},\ldots,m_{N})\leq 2\sum_{j=2}^{N}\sum_{i=1}^{j-1}\Big(m_{i}\,U^{\mu^{*}}_{G}(a_{i})+\int f\,d\mu^{*}+m_{i}\,f(a_{i})\Big)+2\sum_{k=1}^{N}m_{k}\,f(a_{k}).
\]
Observe that $U_{G}^{\mu^{*}}(a_{i})+f(a_{i})\leq C_{f}=0$, and $v_{f}=\int f\,d\mu^{*}$, see Theorems \ref{theo:fund} and  \ref{theo:basicprop}. So from the previous inequality we obtain
\begin{equation}\label{ineqENfgreedy}
E_{N,f}(a_{1},\ldots,a_{N},m_{1},\ldots,m_{N})\leq (N-1)N\,v_{f}+2\sum_{k=1}^{N}m_{k}\,f(a_{k}).
\end{equation}
Since $f$ is bounded on $A$, and $e_{N}\leq E_{N,f}(a_{1},\ldots,a_{N},m_{1},\ldots,m_{N})$ for all $N$, from \eqref{ineqENfgreedy} and \eqref{eq:limeNvf} we deduce that \eqref{asympenergygreedy} holds.

Let $\eta_{N}:=\frac{1}{N}\sum_{k=1}^{N}m_{k}\,\delta_{a_{k}}$. Arguing as in the proof of Theorem~\ref{theo:fund}, if $(\eta_{N_{k}})$ is a convergent subsequence with limit $\tau$, then from \eqref{asympenergygreedy} we deduce that $J_{f}(\tau)=v_{f}$, and therefore $\tau=\mu^{*}$. Since every convergent subsequence converges to $\mu^{*}$, the entire sequence $(\eta_{N})$ converges to $\mu^{*}$.

By the lower semicontinuity of $f$ and $U_{G}^{\mu^{*}}$, we get
\begin{align*}
\int f\,d\mu^{*} & \leq \liminf_{N\rightarrow\infty}\int f\,d\eta_{N},\\
\limsup_{N\rightarrow\infty}\int -U_{G}^{\mu^{*}}\,d\eta_{N} & \leq \int -U_{G}^{\mu^{*}}\,d\mu^{*}.
\end{align*}
On $A$ we have $f\leq -U_{G}^{\mu^{*}}$, hence
\[
\limsup_{N\rightarrow\infty}\int f\,d\eta_{N}\leq \int -U^{\mu^{*}}_{G}\,d\mu^{*}=\int f\,d\mu^{*},
\]
where in the equality we used that $f=-U^{\mu^{*}}_{G}$ holds $\mu^{*}$-almost everywhere. This is a consequence of the fact that $f=-U^{\mu^{*}}_{G}$ q.e. on $\supp(\mu^{*})$ (see \eqref{varineq1} and \eqref{varineq2}), $\mu^{*}$ has finite logarithmic energy, and so by Theorem 3.2.3 in \cite{Rans}, the $\mu^{*}$-measure of the set $\{z\in \supp(\mu^{*}): f(z)\neq -U_{G}^{\mu^{*}}(z)\}$ must be zero. So we conclude that 
\[
\lim_{N\rightarrow\infty}\int f\,d\eta_{N}=\lim_{N\rightarrow\infty}\frac{1}{N}\sum_{k=1}^{N}m_{k}\,f(a_{k})=\int f\,d\mu^{*}.
\]
Then \eqref{asympgas} follows from \eqref{asympfs} and \eqref{asympenergygreedy}. 

Fix $N\geq 2$ and let $\chi(z)=\sum_{k=1}^{N-1}(m_{k}\,g(z,a_{k})+f(z))$. Then $U_{N,f}(z,m)=m\,\chi(z)$. If $\min_{z\in A}\chi(z)>0$, then clearly $m_{N}=0$, and if $\min_{z\in A}\chi(z)<0$, then necessarily $m_{N}=R$. If $\min_{z\in A}\chi(z)=0$, then one can choose $m_{N}=0$.
\end{proof}

Note that in Theorem~\ref{theo:greedy} we have $v_{f}<0$, since $v_{f}=0$ would imply that $\mu^{*}$ is the zero measure. 

Finally, we remark that following the argument in the proof of equation (2.5) in \cite{Lop}, one can show that under the same hypothesis of Theorem~\ref{theo:greedy}, we have
\[
\lim_{N\rightarrow\infty}\frac{U_{N,f}(a_{N},m_{N})}{N}=0.
\]
It is then natural to expect that $U_{N,f}(a_{N},m_{N})=\mathcal{O}(\log N)$. If so, an interesting problem would be the study of the asymptotic behavior of $U_{N,f}(a_{N},m_{N})/\log N$ in case that this is the proper normalization of $U_{N,f}(a_{N},m_{N})$.

\section{Constrained variational problem, and discrete energy with prescribed particle positions}\label{sec:constrained}

In this section we discuss the relation between the continuous energy problem with upper constraint and the discrete energy problem with prescribed particle positions.

If $h:K\rightarrow\mathbb{R}\cup\{+\infty\}$ is a function and $\sigma$ is a positive measure on $K$, recall 
\[
\mathrm{ess\,inf}_{\sigma}\,h=\sup\,\{r\in\mathbb{R}: h(x)\geq r\,\,\mbox{holds}\,\,\sigma\,\,\mbox{a.e. on}\,\,K\}.
\]
The essential infimum is understood to be $-\infty$ if the set of essential lower bounds is empty.

\begin{theorem}\label{theo:fund:constr}
Let $\lambda$ be a probability measure on $K$ with finite Green energy. There exists a unique measure $\mu_{\lambda}^{*}$ satisfying $0\leq \mu_{\lambda}^{*}\leq R\lambda$ and $J_{f}(\mu_{\lambda}^{*})=v_{f,\lambda}=\inf\{J_{f}(\mu): 0\leq \mu\leq R\lambda\}$. Let $\rho_{\lambda}:=R\lambda-\mu_{\lambda}^{*}$ and $S_{\lambda}:=\supp(\mu_{\lambda}^{*})$. If $\|\mu_{\lambda}^{*}\|>0$, then $v_{f,\lambda}<0$ and
\begin{equation}\label{newvarineq}
\sup_{z\in S_{\lambda}}\,(U^{\mu_{\lambda}^{*}}_{G}(z)+f(z))\leq \mathrm{ess\,inf}_{\rho_{\lambda}}\,(U^{\mu_{\lambda}^{*}}_{G}(z)+f(z)).
\end{equation}
\end{theorem}
\begin{proof}
By \eqref{ineqvfvfl}, the constant $v_{f,\lambda}$ is finite. For each $N\geq 1$, let $\tau_{N}$ be a measure satisfying
\[
J_{f}(\tau_{N})< v_{f,\lambda}+\frac{1}{N}
\]
and $0\leq \tau_{N}\leq R\lambda$. Let $(\tau_{N_{k}})$ be a subsequence that converges in the weak-star topology to a measure $\tau\in \mathcal{B}_{R}(K)$. Then, as argued in the proof of Theorem~\ref{theo:dynamic:constr} below, we have $\tau\leq R\lambda$. By the lower semicontinuity of $J_{f}$, 
\[
J_{f}(\tau)\leq \liminf_{k\rightarrow\infty}J_{f}(\tau_{N_{k}})=v_{f,\lambda},
\]
hence $J_{f}(\tau)=v_{f,\lambda}$, so a minimizer exists. The space $\{\mu:0\leq \mu\leq R\lambda\}$ is convex, so the uniqueness of the minimizer follows as in the proof of  Theorem~\ref{theo:fund}. If $\|\mu_{\lambda}^{*}\|>0$ and $v_{f,\lambda}=0$, then $\mu_{\lambda}^{*}$ and the zero measure are two different minimizers of $J_{f}$, which is impossible.

For the proof of \eqref{newvarineq}, we follow closely the argument in the proof of Theorem 2.1(c) in \cite{DragnevSaff}. Assume that \eqref{newvarineq} is false, and let $r_{1}, r_{2}$ be finite constants such that
\[
\sup_{z\in S_{\lambda}}\,(U^{\mu_{\lambda}^{*}}_{G}(z)+f(z))>r_{1}>r_{2}>\mathrm{ess\,inf}_{\rho_{\lambda}}\,(U^{\mu_{\lambda}^{*}}_{G}(z)+f(z)).
\]
Let $z_{0}\in S_{\lambda}$ be a point where $U^{\mu_{\lambda}^{*}}_{G}(z_{0})+f(z_{0})>r_{1}$. Since $U^{\mu_{\lambda}^{*}}_{G}+f$ is lower semicontinuous, there exists $\epsilon>0$ such that 
\[
U^{\mu_{\lambda}^{*}}_{G}(z)+f(z)>r_{1}\quad\mbox{for all}\,\,z\in K_{1}:=\{z: |z-z_{0}|<\epsilon\}\cap S_{\lambda}.
\]
Clearly, $\mu_{\lambda}^{*}(K_{1})>0$. We can also find a set $K_{2}\subset \supp(\rho_{\lambda})$ with $\rho_{\lambda}(K_{2})>0$ such that
\[
U^{\mu_{\lambda}^{*}}_{G}(z)+f(z)<r_{2}\quad\mbox{for all}\,\,z\in K_{2}.
\]
So $K_{1}\cap K_{2}=\emptyset$. Choose constants $0<\alpha<1$ and $0<\beta<1$ such that $\alpha\,\mu_{\lambda}^{*}(K_{1})=\beta\, \rho_{\lambda}(K_{2})$. Let $\eta$ be the signed measure defined to be $-\alpha\, \mu_{\lambda}^{*}$ on $K_{1}$, $\beta\rho_{\lambda}$ on $K_{2}$, and zero elsewhere. Then it is easy to see that for any value of $0<\delta<1$, the measure $\mu_{\lambda}^{*}+\delta\eta$ satisfies $0\leq \mu_{\lambda}^{*}+\delta\eta\leq R\lambda$. Since $\lambda$ has finite Green energy (cf. P4)) and $f$ is bounded on $K_{2}$, the measure $\eta$ also has finite Green energy and the integral $J_{f}(\mu_{\lambda}^{*}+\delta\eta)$ is well-defined and finite. Then one can easily check that
\begin{align*}
J_{f}(\mu^{*}_{\lambda}+\delta\eta)-J_{f}(\mu_{\lambda}^{*}) & =2\delta\int(U^{\mu_{\lambda}^{*}}_{G}+f)\,d\eta+\delta^{2}\iint g(z,t)\,d\eta(z)\,d\eta(t)\\
& \leq 2\delta\,(r_{2}\,\beta\,\rho_{\lambda}(K_{2})-r_{1}\,\alpha\, \mu_{\lambda}^{*}(K_{1}))+\delta^{2}\iint g(z,t)\,d\eta(z)\,d\eta(t)\\
& =-2\delta\,\alpha\,(r_1-r_2)\mu_{\lambda}^{*}(K_1)+\delta^{2}\iint g(z,t)\,d\eta(z)\,d\eta(t)<0
\end{align*}
for all $\delta>0$ sufficiently small. Hence $J_{f}(\mu^{*}_{\lambda}+\delta\eta)<J_{f}(\mu_{\lambda}^{*})$, which is a contradiction.
\end{proof}

\begin{remark}
If $\sup_{z\in S_{\lambda}}\,(U^{\mu_{\lambda}^{*}}_{G}(z)+f(z))$ is finite, then from \eqref{newvarineq} we deduce that there exists a constant $C_{f,\lambda}$ such that
\begin{align*}
U^{\mu_{\lambda}^{*}}_{G}(z)+f(z) & \geq C_{f,\lambda}\quad \mbox{holds}\,\,(R\lambda-\mu_{\lambda}^{*})\,\,\mbox{a.e. on}\,\,K,\\
U^{\mu_{\lambda}^{*}}_{G}(z)+f(z) & \leq C_{f,\lambda}\quad \mbox{holds for all}\,\,z\in S_{\lambda}=\supp(\mu_{\lambda}^{*}).
\end{align*}
Compare this with \eqref{varineq1} and \eqref{varineq2}.
\end{remark}

\begin{remark}
In the problem analyzed in Theorem \ref{theo:fund:constr}, one can easily show that the condition $f\geq 0$ q.e. on $K$ implies that $\mu_{\lambda}^{*}=0$ is the zero measure, but the converse is not true (see for comparison Theorem \ref{theo:basicprop}, $i)$). Indeed, taking $K=[0,1]\cup[2,3]$, $\lambda$ the probability measure given by $\lambda|_{[0,1]}=dx$, $\lambda|_{[2,3]}=0$, and
\[
f(x)=\begin{cases}
0, & x\in[0,1],\\
-1, & x\in[2,3],
\end{cases}
\]
provides a counterexample for any $R>0$.    
\end{remark}

\begin{theorem}\label{theo:dynamic:constr}
Let $\mu_{\lambda,R}^{*}$ be the unique measure satisfying $0\leq \mu_{\lambda,R}^{*}\leq R\lambda$ and $J_{f}(\mu_{\lambda,R}^{*})=v_{f,\lambda}(R):=\inf\{J_{f}(\mu): 0\leq \mu\leq R\lambda\}$. If $(R_{N})_{N}$ is a convergent sequence of positive numbers with limit $R$, then $\mu_{\lambda,R_{N}}^{*}$ converges in the weak-star topology to $\mu_{\lambda,R}^{*}$ and $v_{f,\lambda}(R)=\lim_{N\rightarrow\infty} v_{f,\lambda}(R_{N})$.
\end{theorem}
\begin{proof}
For simplicity of notation, we write $\mu_{N}^{*}$ instead of $\mu_{\lambda,R_{N}}^{*}$. To prove the convergence of $\mu_{N}^{*}$ to $\mu_{\lambda,R}^{*}$, we show that any convergent subsequence of $(\mu_{N}^{*})$ has limit $\mu_{\lambda,R}^{*}$. Let $(\mu_{N_{k}}^{*})$ be a weak-star convergent subsequence with limit $\tau$. We have $\|\tau\|=\lim_{k\rightarrow\infty}\|\mu_{N_{k}}^{*}\|\leq \lim_{k\rightarrow\infty} R_{N_{k}}=R$, so $\tau\in \mathcal{B}_{R}(K)$. Since $\mu_{N_{k}}^{*}\leq R_{N_{k}}\lambda$, for every $h\in C(K)$, $h\geq 0$, we have
\[
\int h\,d(R\lambda-\tau)=\lim_{k\rightarrow\infty}\int h\,d(R_{N_{k}}\lambda-\mu_{N_{k}}^{*})\geq 0.
\]
This implies by the Riesz representation theorem that $R\lambda-\tau\geq 0$. Since $\tau\in \mathcal{B}_{R}(K)$ and $\tau\leq R\lambda$, we obtain $J_{f}(\mu_{\lambda,R}^{*})\leq J_{f}(\tau)$. By the lower semicontinuity of $J_{f}$ in the weak-star topology, we have $J_{f}(\tau)\leq \liminf_{k\rightarrow\infty} J_{f}(\mu_{N_{k}}^{*})$. Let $\widehat{\mu}_{k}:=(R_{N_{k}}/R)\,\mu_{\lambda,R}^{*}$. Then we have $\widehat{\mu}_{k}\leq R_{N_{k}}\lambda$, and so $J_{f}(\mu_{N_{k}}^{*})\leq J_{f}(\mu_{k}^{*})$. The rest of the argument is the same as in the proof of Theorem~\ref{theo:dynamic}.\end{proof}

In the rest of this section, we fix a sequence of point configurations
\[
(x_{j,N})_{1\leq j\leq l_{N}},\quad N\geq 1,\quad l_{N}\geq 2,
\]
in the compact set $K\subset D\setminus\{\infty\}$, where $l_{N}/N\rightarrow 1$, satisfying the properties $\textrm{P1)}$--$\textrm{P4)}$ indicated in the introduction. Let $\lambda$ be the limiting probability measure from $\textrm{P1)}$.

\begin{example}\label{examplemain}
Consider the domain $D=\{z\in\mathbb{C}: \mathrm{Re}(z)>0\}$, with Green function $g(z,\zeta)=\log|(z+\overline{\zeta})/(z-\zeta)|$. Let $\{a_{i}\}_{1}^{s}$ and $\{b_{i}\}_{1}^{s}$ be real numbers such that
\[
0<a_{1}<b_{1}<a_{2}<b_{2}<\cdots <a_{s-1}<b_{s-1}<a_{s}<b_{s}.
\]
Take 
\[
K=\bigcup_{i=1}^{s}[a_{i},b_{i}]
\]
and let $d\lambda(x)=\frac{1}{L}\,dx$ on $K$, where $L=\sum_{i=1}^{s}(b_{i}-a_{i})$. For all $N$ large enough, we partition each interval $[a_{i},b_{i}]$ as a union of subintervals, all having $\lambda$-measure $1/N$ (or length $\Delta=L/N$) except possibly the last subinterval, which has $\lambda$-measure less than $1/N$ if $N(b_{i}-a_{i})/L$ is not an integer. Specifically, we use the partition
\[
[a_{i},b_{i}]=[a_{i},a_{i}+\Delta)\cup[a_{i}+\Delta,a_{i}+2\Delta)\cup[a_{i}+2\Delta,a_{i}+3\Delta)\cup\cdots\cup[a_{i}+(\rho_{i}-1)\Delta,b_{i}]
\]
where $\rho_{i}$ denotes the ceiling of $N(b_{i}-a_{i})/L$. We define the sets $V_{j,N}$ as the subintervals obtained in these partitions, and the points $x_{j,N}$ as the middle points of the subintervals. Then, the distance between two consecutive points $x_{j,N}$ is at least $L/2N$. It is clear that the four properties \textrm{P1)}--\textrm{P4)} are satisfied, and all the constructed partitions are monotonic.\end{example}

\begin{example}\label{examplemain1}
Suppose that $\Gamma$ is a rectifiable Jordan arc in the complex plane, i.e., $\Gamma$ is a plane curve that has a continuous and one-to-one parametrization $\gamma:[a,b]\rightarrow\Gamma$ with finite total variation. Let $K\subset \Gamma$ consists of a finite number of disjoint compact subarcs $K_{j}$, $j=1,\ldots,s$, of $\Gamma$, and let $\lambda$ be the arclength measure on $K$ normalized so that $\lambda(K)=1$. In this example we show how to construct monotonic partitions of $K$ that satisfy property $\mathrm{P2)}$. Fix $N\geq 1$, and split $K$ into $N$ successive disjoint subarcs (or unions of disjoint subarcs) $M_j$, such that $\l(M_j)=1/N$, while traversing $K$ in the positive direction. It is clear that ${\rm diam}  (M_j)\leq \frac LN$ for every $M_j$ that lies on a single subcompact $K_m$. If $M_j$ has nontrivial (positive measure) intersection with $k>1$ subcompacts $K_{m_1}, \dots, K_{m_k}$, we split it into  $k$ subarcs $M_j\cap  K_{m_i}$, $i=1,\dots,k$. The obtained partition $\mathcal{P}_N$ of $K$ consists of $l_N\leq N+s$ elements of partition $V_{j,N}$, each with $\lambda$-measure not exceeding $\frac 1N$ and with diameter not exceeding $\frac LN$. Thus, the constructed partition satisfies property $\mathrm{P2)}$. The monotonicity of the partition is clear from the construction.
\end{example}

\begin{lemma}\label{prelimlem}
Let $\mu\in \mathcal{B}_{R}(K)$. Then $\mu\leq R\lambda$ if and only if there exist numbers $0\leq m_{j,N}\leq R$, $1\leq j\leq l_N$, such that the sequence of discrete measures $\mu_{N}=\frac{1}{l_N}\sum_{j=1}^{l_N} m_{j,N}\,\delta_{x_{j,N}}$ converges in the weak-star topology to $\mu$ as $N$ tends to infinity.
\end{lemma}
\begin{proof}
Assume first that $\mu$ is the weak-star limit of $\mu_{N}=\frac{1}{l_{N}}\sum_{j=1}^{l_{N}} m_{j,N}\,\delta_{x_{j,N}}$, with $0\leq m_{j,N}\leq R$ for all $j$. Then
\[
\frac{1}{R}\mu_{N}=\frac{1}{R\,l_{N}}\sum_{j=1}^{l_N}m_{j,N}\,\delta_{x_{j,N}}\leq \frac{1}{l_{N}}\sum_{j=1}^{l_{N}}\delta_{x_{j,N}}=\lambda_{N}.
\]
So for every $h\in C(K)$, $h\geq 0$, we have
\[
\int h\,d(\lambda_{N}-\frac{1}{R}\,\mu_{N})\geq 0
\]
and so
\[
0\leq \int h\,d(\lambda-\frac{1}{R}\,\mu)=\lim_{N\rightarrow\infty}\int h\,d(\lambda_{N}-\frac{1}{R}\,\mu_{N}).
\]
Let $\eta=\lambda-\frac{1}{R}\mu$. Since $\eta$ is a finite measure, the map $h\mapsto \int h\,d\eta$ is a bounded and positive linear functional on the space $C(K)$. By the Riesz representation theorems for such functionals, we deduce that $\eta\geq 0$, i.e., $\mu\leq R\lambda$. 

Now assume that $\tau:=\frac{\mu}{R}\leq \lambda$, where $\mu\in \mathcal{B}_{R}(K)$. Then, by property $\mathrm{P2)}$, we have $\tau(V_{j,N})\leq \lambda(V_{j,N})\leq 1/N$ for every $1\leq j\leq l_N$. Let $m_{j,N}:=N\mu(V_{j,N})=RN \tau(V_{j,N})\in[0,R]$, $1\leq j\leq l_N$, and let us show that $\mu$ is the weak-star limit of the sequence of measures $\mu_{N}=\frac{1}{l_{N}}\sum_{j=1}^{l_{N}}m_{j,N}\,\delta_{x_{j,N}}$. Let $h\in C(K)$, with modulus of continuity 
\[
\omega(h;t):=\sup\{|h(x)-h(y)|: |x-y|\leq t,\,\,x,y\in K\}
\]
and let $\|h\|_{K}=\max_{x\in K}|h(x)|$.

We have
\begin{align*}
\left|\int_{K} h\,d\mu_{N}-\int_{K} h\,d\mu\right| & =\left|\sum_{j=1}^{l_{N}}\frac{m_{j,N}}{l_{N}}\,h(x_{j,N})-\sum_{j=1}^{l_N}\int_{V_{j,N}} h\,d\mu\right|\\
& =\left|\frac{N}{l_{N}}\sum_{j=1}^{l_{N}}\frac{m_{j,N}}{N} h(x_{j,N})-\frac{N}{l_{N}}\sum_{j=1}^{l_{N}}\int_{V_{j,N}}\frac{l_{N}}{N}\,h\,d\mu\right|\\
& =\frac{N}{l_{N}}\left|\sum_{j=1}^{l_N}\int_{V_{j,N}}(h(x_{j,N})-\frac{l_{N}}{N}\,h(z))\,d\mu(z)\right|\\
& \leq \frac{N}{l_{N}}\sum_{j=1}^{l_N}\int_{V_{j,N}}|h(x_{j,N})-\frac{l_{N}}{N}\,h(z)|\,d\mu(z).
\end{align*}
Recall that $x_{j,N}\in V_{j,N}$ for all $j$. Therefore for $z\in V_{j,N}$ we have
\begin{align*}
|h(x_{j,N})-\frac{l_{N}}{N}\,h(z)| & \leq |h(x_{j,N})-h(z)|+|1-\frac{l_{N}}{N}||h(z)|\\& \leq \omega(h;\kappa_{N})+|1-\frac{l_{N}}{N}|\|h\|_{K}
\end{align*}
where we used $\mathrm{diam}(V_{j,N})\leq \kappa_{N}$. Since $\mu(V_{j,N})\leq R/N$, we conclude that
\begin{align*}
\left|\int_{K} h\,d\mu_{N}-\int_{K} h\,d\mu\right| & \leq \frac{N}{l_{N}}\sum_{j=1}^{l_{N}}(\omega(h;\kappa_{N})+|1-\frac{l_{N}}{N}|\|h\|_{K})\,\mu(V_{j,N})\\
& \leq R(\omega(h;\kappa_{N})+|1-\frac{l_{N}}{N}|\|h\|_{K})
\end{align*}
and this expression approaches zero as $N$ tends to infinity.
\end{proof}

\begin{remark}
Properties \textrm{P3)} and \textrm{P4)} are clearly not needed in Lemma~\ref{prelimlem}.
\end{remark}

Recall the definitions of the constants $v_{f,\lambda}$ and $d_{N}$, see \eqref{def:vflambda} and \eqref{defdN}. 

\begin{proposition}
We have
\begin{equation}\label{liminfdN}
\liminf_{N\rightarrow\infty}\frac{d_{N}}{N^{2}}\geq v_{f,\lambda}. 
\end{equation}
\end{proposition}
\begin{proof}
For each $N\geq 1$, let $\{\widehat{m}_{j,N}\}_{1\leq j\leq l_N}\subset [0,R]$ be such that
\[
d_{N}=E_{l_{N},f}(x_{1,N},\ldots,x_{l_{N},N},\widehat{m}_{1,N},\ldots,\widehat{m}_{l_{N},N})
\]
holds, and let
\[
\mu_{N}:=\frac{1}{l_{N}}\sum_{j=1}^{l_{N}} \widehat{m}_{j,N}\,\delta_{x_{j,N}}.
\]
By \eqref{eq:limeNvf} and \eqref{ineqeNdN}, the sequence $(\frac{d_{N}}{N^{2}})$ is bounded. Let $(\frac{d_{N_{k}}}{N_{k}^2})$ be a convergent subsequence. The corresponding sequence $(\mu_{N_{k}})$ has a weak-star convergent subsequence with limit $\mu\in \mathcal{B}_{R}(K)$, and by Lemma~\ref{prelimlem} we also have $\mu\leq R\lambda$. For simplicity, the convergent subsequence of $(\mu_{N_{k}})$ will be denoted again as $(\mu_{N_{k}})$. If $g_{M}$ is the truncated Green kernel \eqref{deftruncGK}, then arguing as in \eqref{halfestimate} we obtain
\[
\iint g_{M}(z,\zeta)\,d\mu(z)\,d\mu(\zeta)+2\int f\,d\mu\leq \lim_{k\rightarrow\infty}\frac{d_{N_{k}}}{N_{k}^2},
\]
and letting $M\rightarrow\infty$ we get
\[
v_{f,\lambda}\leq J_{f}(\mu)\leq \lim_{k\rightarrow\infty}\frac{d_{N_{k}}}{N_{k}^2}.
\]
\end{proof}

\begin{remark}
By the same argument used to prove $vii)$ in Theorem \ref{theo:basicprop}, we can say that if $(\widehat{m}_{1,N},\widehat{m}_{2,N},\ldots,\widehat{m}_{l_{N},N})\in[0,R]^{l_{N}}$ is an optimal configuration of masses for \eqref{defdN}, and $0<\widehat{m}_{i,N}<R$ for some $1\leq i\leq l_{N}$, then $f(x_{i,N})$ is finite and 
\[
\sum_{\substack{j=1\\ j\neq i}}^{l_{N}}\widehat{m}_{j,N}\,g(x_{i,N},x_{j,N})=-l_{N}\,f(x_{i,N}).
\]
Thus, if $0<\widehat{m}_{i,N}<R$ for all $1\leq i\leq l_{N}$, then for the vector $\widehat{\mathbf{v}}:=(\widehat{m}_{1,N}\,\,\,\widehat{m}_{2,N}\,\,\,\ldots\,\,\widehat{m}_{l_N,N})^{t}$ we have
\[
G_{N}\,\widehat{\mathbf{v}}=\widehat{F}_{N}
\] 
where $G_{N}$ is the $l_N\times l_N$ symmetric matrix with entries 
\[
G_{N}(i,j)=\begin{cases}
g(x_{i,N},x_{j,N}), & i\neq j,\\
0, & i=j,
\end{cases}
\]
and $\widehat{F}_{N}$ is the column vector $\widehat{F}_{N}=-N (f(x_{1,N})\,\,\,f(x_{2,N})\,\,\,\ldots\,\,\,f(x_{l_N,N}))^{t}$. Thus, if $G_{N}$ is invertible then $\widehat{\mathbf{v}}=G_N^{-1}\widehat{F}_{N}$. 
\end{remark}

The following three lemmas are preparatory for the proof of Theorem~\ref{theo:main:upcon}. 

\begin{lemma}
The function
\[
h(z,t):=\begin{cases}
g(z,t)-\log\frac{1}{|z-t|} & \mbox{if}\,\,z\neq t\\
\lim_{w\rightarrow t}\,(g(w,t)-\log\frac{1}{|w-t|}) & \mbox{if}\,\,z=t
\end{cases}
\]
is continuous on $(D\setminus\{\infty\})\times(D\setminus\{\infty\})$.
\end{lemma}
\begin{proof}
Assume first that $z_{0}, t_{0}$ are distinct points in $D\setminus\{\infty\}$. Let $r>0$ be such that the disks $B(z_{0},r)$ and $B(t_{0},r)$ are disjoint and are contained in $D$, and let us show that $g(z,t)\rightarrow g(z_{0},t_{0})$ as $(z,t)\rightarrow(z_{0},t_{0})$. It is obvious that $g(z_{0},t)=g(t,z_{0})\rightarrow g(t_{0},z_{0})=g(z_{0},t_{0})$ as $t\rightarrow t_{0}$. For a fixed $t\in B(t_{0},r)$, applying Harnack's inequality to the function $z\mapsto g(z,t)$ on $B(z_{0},r)$, we get
\[
\frac{r-|z-z_{0}|}{r+|z-z_{0}|}\,g(z_{0},t)\leq g(z,t)\leq \frac{r+|z-z_{0}|}{r-|z-z_{0}|}\, g(z_{0},t),\qquad z\in B(z_{0},r),
\] 
and from this we obtain that $g(z,t)-g(z_{0},t)\rightarrow 0$ as $(z,t)\rightarrow(z_{0},t_{0})$. This proves the continuity of $h$ at $(z_{0},t_{0})$ when $z_{0}\neq t_{0}$.

Assume now that $z_{0}=t_{0}$. Fix $t$ such that $|t-z_{0}|<\frac{r}{2}$. The function $z\mapsto h(z,t)$ is harmonic in $D\setminus\{\infty\}$, so
\[
\inf_{z\in B(z_{0},r)}h(z,t)=h(z_{t},t)
\]
for some $z_{t}$ such that $|z_{t}-z_{0}|=r$. Therefore
\[
\inf_{z\in B(z_{0},r)} h(z,t)=g(z_{t},t)+\log|z_{t}-t|>\log(r/2).
\] 
Applying Harnack's inequality to the positive harmonic function $z\mapsto h(z,t)-\log(r/2)$  on $B(z_{0},r)$, we deduce that
\begin{equation}\label{Harnackineq}
-\frac{2|z-z_{0}|}{r+|z-z_{0}|}(h(z_{0},t)-\log(r/2))\leq h(z,t)-h(z_{0},t)\leq \frac{2|z-z_{0}|}{r-|z-z_{0}|}(h(z_{0},t)-\log(r/2))
\end{equation}
for all $z\in B(z_{0},r)$, $t\in B(z_{0},r/2)$. It is clear that $h(z_{0},t)\rightarrow h(z_{0},z_{0})$ as $t\rightarrow z_{0}$, so \eqref{Harnackineq} implies the continuity of $h$ at $(z_{0},z_{0})$.
\end{proof}

We deduce from the previous result that for any compact set $B\subset D\setminus\{\infty\}$, the Green function $g(z,\zeta)$ for $D$ admits the form
\begin{equation}\label{repGreenfunc}
g(z,\zeta)=\log\frac{1}{|z-\zeta|}+h(z,\zeta),\qquad (z,\zeta)\in B\times B,
\end{equation}
where $h(z,\zeta)$ is continuous on $B\times B$.  

\begin{lemma}\label{lemmasuffice}
Assume that $K\subset D\setminus\{\infty\}$ and the external field $f$ is continuous on $K$. Suppose that for each $N\geq 1$, there exist constants $\{m_{j,N}\}_{j=1}^{l_{N}}\subset[0,R]$ such that the sequence of measures $\mu_{N}:=\frac{1}{l_{N}}\sum_{j=1}^{l_{N}} m_{j,N}\,\delta_{x_{j,N}}$ converges in the weak-star topology to the minimizer $\mu_{\lambda}^{*}$ and we have
\begin{equation}\label{hypothesislim}
\lim_{N\rightarrow\infty}\frac{1}{l_{N}^2}\sum_{1\leq i\neq j\leq l_{N}} m_{i,N}\,m_{j,N}\,\log\frac{1}{|x_{i,N}-x_{j,N}|}=\iint\log\frac{1}{|z-\zeta|}\,d\mu_{\lambda}^{*}(z)\,d\mu_{\lambda}^{*}(\zeta).
\end{equation} 
Then
\begin{equation}\label{asympdN}
\lim_{N\rightarrow\infty}\frac{d_{N}}{N^2}=v_{f,\lambda}.
\end{equation}
Additionally, if $\{\widehat{m}_{j,N}\}_{j=1}^{l_{N}}\subset[0,R]$ is a collection of numbers such that 
\begin{equation}\label{attaindN}
d_{N}=E_{l_{N},f}(x_{1,N},\ldots,x_{l_{N},N},\widehat{m}_{1,N},\ldots,\widehat{m}_{l_{N},N}),\qquad\mbox{for all}\,\,\,\,N\geq 1,
\end{equation}
then the sequence of measures $\frac{1}{l_{N}}\sum_{j=1}^{l_{N}}\widehat{m}_{j,N}\,\delta_{x_{j,N}}$ converges in the weak-star topology to $\mu_{\lambda}^{*}$.
\end{lemma}
\begin{proof}
In virtue of \eqref{repGreenfunc}, we have
\begin{equation}\label{decompGreen}
g(z,\zeta)=\log\frac{1}{|z-\zeta|}+h(z,\zeta), \qquad (z,\zeta)\in K\times K,
\end{equation}
where $h(z,\zeta)$ is continuous. Hence
\[
\lim_{N\rightarrow\infty}\iint h(z,\zeta)\,d\mu_{N}(z)\,d\mu_{N}(\zeta)=\iint h(z,\zeta)\,d\mu_{\lambda}^{*}(z)\,d\mu_{\lambda}^{*}(\zeta),
\]
which is clearly equivalent to
\begin{equation}\label{limpart1}
\lim_{N\rightarrow\infty}\frac{1}{l_{N}^{2}}\sum_{1\leq i\neq j\leq l_{N}}m_{i,N}\,m_{j,N}\,h(x_{i,N},x_{j,N})=\iint h(z,\zeta)\,d\mu_{\lambda}^{*}(z)\,d\mu_{\lambda}^{*}(\zeta).
\end{equation}
Since $f$ is continuous on $K$, we have
\begin{equation}\label{limpart2}
\lim_{N\rightarrow\infty}\int f\,d\mu_{N}=\lim_{N\rightarrow\infty}\frac{1}{l_{N}}\sum_{i=1}^{l_{N}}m_{i,N}\,f(x_{i,N})=\int f\,d\mu_{\lambda}^{*}.
\end{equation}
We deduce from \eqref{decompGreen}, \eqref{hypothesislim}, \eqref{limpart1}, and \eqref{limpart2} that
\begin{equation}\label{keylimit}
\lim_{N\rightarrow\infty}\frac{E_{l_{N},f}(x_{1,N},\ldots,x_{l_{N},N},m_{1,N},\ldots,m_{l_{N},N})}{N^2}=J_{f}(\mu_{\lambda}^{*})=v_{f,\lambda}.
\end{equation}
Since $d_{N}\leq E_{l_{N},f}(x_{1,N},\ldots,x_{l_{N},N},m_{1,N},\ldots,m_{l_{N},N})$, from \eqref{keylimit} we get $\limsup_{N\rightarrow\infty} d_{N}/N^2\leq v_{f,\lambda}$, which together with \eqref{liminfdN} implies \eqref{asympdN}.

Now assume that \eqref{attaindN} holds for some numbers $\{\widehat{m}_{j,N}\}_{j=1}^{l_{N}}\subset [0,R]$, and let us show that $\tau_{N}:=\frac{1}{l_{N}}\sum_{j=1}^{l_{N}}\widehat{m}_{j,N}\,\delta_{x_{j,N}}$ converges to $\mu_{\lambda}^{*}$. Let $\tau$ be a weak-star limit point of the sequence $(\tau_{N})$, and let $(\tau_{N_{k}})$ be a subsequence converging to $\tau$. Then, for each $M>0$ fixed, since $f$ is continuous, arguing as in \eqref{estimate} we obtain
\begin{gather*}
\iint g_{M}(z,\zeta)\,d\tau(z)\,d\tau(\zeta)+2\int f\,d\tau=\lim_{k\rightarrow\infty}\left(\iint g_{M}(z,\zeta)\,d\tau_{N_{k}}(z)\,d\tau_{N_{k}}(\zeta)+2\int f\,d\tau_{N_{k}}\right)\\
\leq \lim_{k\rightarrow\infty}\left(\frac{E_{l_{N_{k}},f}(x_{1,N_{k}},\ldots,x_{l_{N_{k}},N_{k}},\widehat{m}_{1,N_{k}},\ldots,\widehat{m}_{l_{N_{k}},N_{k}})}{N_{k}^2}+\mathcal{O}(N_{k}^{-1})\right)=\lim_{k\rightarrow\infty}\frac{d_{N_{k}}}{N_{k}^2}=v_{f,\lambda}.
\end{gather*}
So 
\[
\iint g_{M}(z,\zeta)\,d\tau(z)\,d\tau(\zeta)+2\int f\,d\tau\leq v_{f,\lambda},\qquad\mbox{for all}\,\,M>0.
\]
Letting $M\rightarrow\infty$ we get $J_{f}(\tau)\leq v_{f,\lambda}$, so $J_{f}(\tau)=v_{f,\lambda}$. By the uniqueness of the minimizer, we obtain $\tau=\mu_{\lambda}^{*}$. Since any weak-star limit point of $(\tau_{N})$ equals $\mu_{\lambda}^{*}$, the sequence $(\tau_{N})$ converges to $\mu_{\lambda}^{*}$. 
\end{proof}

\begin{lemma}\label{lemreferee}
If the logarithmic potential $U^{\lambda}$ of the reference measure $\lambda$ in property $\mathrm{P1)}$ is continuous on $K$, then we have
\begin{equation}\label{specialassump}
\lim_{\epsilon\rightarrow 0}\sup_{x\in K}\int_{\{t\in K: |t-x|<\epsilon\}}\log\frac{1}{|x-t|}\,d\lambda(t)=0.
\end{equation}
\end{lemma}
\begin{proof}
Assume that $U^{\lambda}$ is continuous on $K$ and \eqref{specialassump} does not hold. Then there exist $\delta>0$ and sequences $\epsilon_{n}\downarrow 0$ and $(x_{n})_{n=1}^{\infty}\subset K$ such that
\begin{equation}\label{eq:ineqdelta}
\int_{\{t\in K: |t-x_{n}|<\epsilon_{n}\}}\log\frac{1}{|x_{n}-t|}\,d\lambda(t)\geq \delta,\qquad \mbox{for all}\,\,n\geq 1.
\end{equation}
We may assume without loss of generality that the sequence $(x_{n})$ converges to a point $x\in K$. Consider the sequence of positive measures $\lambda_{n}$, obtained by restricting $\lambda$:
\[
\lambda_{n}=\lambda|_{\{t\in K: |t-x_{n}|\geq \epsilon_{n}\}}.
\]
For any continuous function $h\in C(K)$, we have
\begin{equation}\label{estweakconv}
\left|\int_{K} h\,d\lambda_{n}-\int_{K} h\,d\lambda\right|\leq \|h\|_{K}\,\lambda(\{t\in K: |t-x|\leq \epsilon_{n}+|x_{n}-x|\}).
\end{equation}
Since $\lambda$ cannot have an atom at $x$ due to the continuity of $U^{\lambda}$, it follows that the right-hand side of \eqref{estweakconv} approaches zero, and so $\lambda_{n}$ converges to $\lambda$ in the weak-star topology.

The condition \eqref{eq:ineqdelta} is equivalent to 
\begin{equation}\label{eq:deltaineq}
U^{\lambda}(x_{n})-U^{\lambda_{n}}(x_{n})\geq \delta,
\end{equation}
and by the principle of descent (cf. Theorem 1.3 in \cite{Landkof}) we have 
\[
\liminf_{n\rightarrow\infty} U^{\lambda_{n}}(x_{n})\geq U^{\lambda}(x),
\]
which contradicts \eqref{eq:deltaineq}.
\end{proof}

\noindent\textbf{Proof of Theorem~\ref{theo:main:upcon}:} From the proof of Lemma~\ref{prelimlem}, we know that if we take $m_{j,N}:=N \mu_{\lambda}^{*}(V_{j,N})$, $1\leq j\leq l_{N}$, then the sequence of measures $$\mu_{N}:=\frac{1}{l_{N}}\sum_{j=1}^{l_{N}} m_{j,N}\,\delta_{x_{j,N}}$$ converges to $\mu_{\lambda}^{*}$ in the weak-star topology. So it follows from Lemma~\ref{lemmasuffice} that \eqref{asympdNnew} and the convergence of $\widehat{\mu}_{N}$ to $\mu_{\lambda}^{*}$ will be justified if we prove 
\begin{equation}\label{goallimit}
\lim_{N\rightarrow\infty}\frac{1}{l_{N}^2}\sum_{1\leq i\neq j\leq l_{N}}m_{i,N}\,m_{j,N}\,\log\frac{1}{|x_{i,N}-x_{j,N}|}=\iint\log\frac{1}{|z-t|}\,d\mu_{\lambda}^{*}(z)\,d\mu_{\lambda}^{*}(t).
\end{equation}
By assumption, the logarithmic potential $U^{\lambda}(z)$ of the reference measure $\lambda$ is continuous on $K$.  Since $\mu_{\lambda}^{*}\leq R\lambda$, this implies by Lemma~5.2 in \cite{DragnevSaff} that $U^{\mu_{\lambda}^{*}}(z)$
is also continuous on $K$. So we have
\[
\lim_{N\rightarrow\infty}\int U^{\mu_{\lambda}^{*}}(z)\,d\mu_{N}(z)=\int U^{\mu_{\lambda}^{*}}(z)\,d\mu_{\lambda}^{*}(z),
\]
equivalently
\begin{align*}
\lim_{N\rightarrow\infty}\frac{1}{l_{N}}\sum_{j=1}^{l_{N}}m_{j,N}\,U^{\mu_{\lambda}^{*}}(x_{j,N}) & =\lim_{N\rightarrow\infty}\frac{1}{l_{N}}\sum_{j=1}^{l_{N}}m_{j,N}\,\int\log\frac{1}{|x_{j,N}-t|}\,d\mu_{\lambda}^{*}(t)\\
& =\iint\log\frac{1}{|z-t|}\,d\mu_{\lambda}^{*}(z)\,d\mu_{\lambda}^{*}(t).
\end{align*} 
The function $\log\frac{1}{|x_{j,N}-t|}$ is integrable with respect to $\mu_{\lambda}^{*}$, and since $V_{1,N},\ldots,V_{l_{N},N}$ is a partition of $K$, we have
\[
\int\log\frac{1}{|x_{j,N}-t|}\,d\mu_{\lambda}^{*}(t)=\sum_{i=1}^{l_{N}}\int_{V_{i,N}}\log\frac{1}{|x_{j,N}-t|}\,d\mu_{\lambda}^{*}(t).
\]
So we have
\begin{equation}\label{partialproof1}
\lim_{N\rightarrow\infty}\frac{1}{l_{N}}\sum_{1\leq i,j\leq l_{N}}m_{j,N}\int_{V_{i,N}}\log\frac{1}{|x_{j,N}-t|}\,d\mu_{\lambda}^{*}(t)=\iint\log\frac{1}{|z-t|}\,d\mu_{\lambda}^{*}(z)\,d\mu_{\lambda}^{*}(t).
\end{equation}
Let
\begin{align*}
u_{N} & :=\frac{1}{l_{N}^2}\sum_{1\leq i\neq j\leq l_{N}}m_{i,N}\,m_{j,N}\,\log\frac{1}{|x_{i,N}-x_{j,N}|},\\
v_{N} & :=\frac{1}{l_{N}}\sum_{1\leq i,j\leq l_{N}}m_{j,N}\int_{V_{i,N}}\log\frac{1}{|x_{j,N}-t|}\,d\mu_{\lambda}^{*}(t).
\end{align*}
Then, from \eqref{partialproof1} we deduce that the proof of \eqref{goallimit} will be finished if we justify the limit
\begin{equation}\label{partialproof2}
\lim_{N\rightarrow\infty} (u_{N}-v_{N})=0.
\end{equation}

We have $K\subset \Gamma$, so it is clear that we may assume that the points $x_{i,N}\in V_{i,N}$ are labeled so that 
\begin{align*}
x_{i,N} & = \gamma(s_{i,N}), \qquad s_{i,N}\in[a,b],\\
s_{i,N} & < s_{i+1,N},\qquad \mbox{for all}\,\,\,1\leq i\leq l_{N}-1. 
\end{align*}
We write
\[
u_{N}=u_{N,1}+u_{N,2}
\]
where
\begin{align*}
u_{N,1} & =\frac{1}{l_{N}^2}\sum_{|i-j|\geq 3}m_{i,N}\, m_{j,N} \log\frac{1}{|x_{i,N}-x_{j,N}|},\\
u_{N,2} & =\frac{1}{l_{N}^2}\sum_{1\leq |i-j|\leq 2}m_{i,N}\, m_{j,N} \log\frac{1}{|x_{i,N}-x_{j,N}|}.
\end{align*}
By property \textrm{P3)}, we have
\[
\log\frac{1}{\mathrm{diam}(K)}\leq \log\frac{1}{|x_{i,N}-x_{j,N}|}\leq \log(N/C),\qquad i\neq j,
\]
which clearly implies
\[
\lim_{N\rightarrow\infty} u_{N,2}=0.
\]
We also use the decomposition
\[
v_{N}=v_{N,1}+v_{N,2}
\]
where
\begin{align*}
v_{N,1} & =\frac{1}{l_{N}}\sum_{|i-j|\geq 3} m_{j,N}\int_{V_{i,N}}\log\frac{1}{|x_{j,N}-t|}\,d\mu_{\lambda}^{*}(t),\\
v_{N,2} & =\frac{1}{l_{N}}\sum_{0\leq |i-j|\leq 2}m_{j,N}\int_{V_{i,N}}\log\frac{1}{|x_{j,N}-t|}\,d\mu_{\lambda}^{*}(t).
\end{align*}
Our next step is to prove
\begin{equation}\label{asympvN2zero}
\lim_{N\rightarrow\infty}v_{N,2}=0.
\end{equation}

First we handle the terms in $v_{N,2}$ with $i=j$. By property P2), $\mathrm{diam}(V_{j,N})\leq \kappa_{N}$ for all $j$ and $\lim_{N\rightarrow\infty}\kappa_{N}=0$. So for all $N$ large enough ($\kappa_{N}<1$ suffices), using $\mu_{\lambda}^{*}\leq R\lambda$ we get 
\begin{align*}
0 & \leq \frac{1}{R}\int_{V_{j,N}}\log\frac{1}{|x_{j,N}-t|}\,d\mu_{\lambda}^{*}(t)\leq\int_{V_{j,N}}\log\frac{1}{|x_{j,N}-t|}\,d\lambda(t)\\
& \leq \sup_{x\in K}\int_{\{t\in K:|t-x|\leq \kappa_{N}\}}\log\frac{1}{|x-t|}d\lambda(t),
\end{align*}
valid for all $1\leq j\leq l_{N}$, so by \eqref{specialassump} we obtain
\begin{equation}\label{partialproof3}
\lim_{N\rightarrow\infty}\frac{1}{l_{N}}\sum_{j=1}^{l_{N}}m_{j,N}\int_{V_{j,N}}\log\frac{1}{|x_{j,N}-t|}\,d\mu_{\lambda}^{*}(t)=0.
\end{equation}

Let $W_{i,N}:=\overline{\mathrm{Co}(\gamma^{-1}(V_{i,N}))}=[\alpha_{i,N},\beta_{i,N}]$ be the closed convex hull of $\gamma^{-1}(V_{i,N})\subset[a,b]$. We claim that $\gamma(\alpha_{i,N}), \gamma(\beta_{i,N})\in \overline{V_{i,N}}\subset K$. Indeed, if $\alpha_{i,N}=\beta_{i,N}$, then $\{\alpha_{i,N}\}=\gamma^{-1}(V_{i,N})\subset\gamma^{-1}(K)$ and the claim is trivially true. Suppose that $\alpha_{i,N}<\beta_{i,N}$. Then for every $\epsilon>0$ such that $\alpha_{i,N}<\alpha_{i,N}+\epsilon<\beta_{i,N}$, we must have $[\alpha_{i,N},\alpha_{i,N}+\epsilon)\cap\gamma^{-1}(V_{i,N})\neq\emptyset$, otherwise we have $\gamma^{-1}(V_{i,N})\subset[\alpha_{i,N}+\epsilon,\beta_{i,N}]$, which contradicts the definition of $W_{i,N}$. So there exists a sequence of points in $\gamma^{-1}(V_{i,N})$ that converges to $\alpha_{i,N}$, therefore $\gamma(\alpha_{i,N})\in \overline{V_{i,N}}\subset K$. In a similar way we show that $\gamma(\beta_{i,N})\in\overline{V_{i,N}}\subset K$. 

Assume that $1\leq|i-j|\leq 2$. Recall that $\gamma(s_{j,N})=x_{j,N}\in V_{j,N}$ and by assumption we have $\mathrm{Co}(\gamma^{-1}(V_{j,N}))\cap\mathrm{Co}(\gamma^{-1}(V_{i,N}))=\emptyset$. So the point $s_{j,N}$ satisfies either $s_{j,N}\leq \alpha_{i,N}$ or $s_{j,N}\geq \beta_{i,N}$. Suppose first that $s_{j,N}\leq \alpha_{i,N}$. Take an arbitrary $s\in W_{i,N}$, and let $y:=\gamma(s)\in\Gamma$. Let $\rho>0$ be the constant in the local monotonicity property of $\Gamma$ (see paragraph before Theorem \ref{theo:main:upcon}). Let $I_{y}$ be the interval satisfying $\gamma(I_{y})=B(y,\rho)\cap\Gamma$. Observe that $s_{j,N}\leq \alpha_{i,N}\leq s$. If $s_{j,N}\notin I_{y}$, then it is easy to see that
\[
|x_{j,N}-y|=|\gamma(s_{j,N})-y|\geq \rho.
\] 
If $s_{j,N}\in I_{y}$, then by the local monotonicity property of the parametrization we have
\[
|x_{j,N}-y|=|\gamma(s_{j,N})-y|\geq |\gamma(\alpha_{i,N})-y|.
\]
So we conclude that
\[
|x_{j,N}-y|\geq \min\{\rho,|\gamma(\alpha_{i,N})-y|\}\qquad \mbox{for all}\,\,y\in V_{i,N}.
\]
This implies 
\[
\log\frac{1}{\mathrm{diam}(K)}\leq \log\frac{1}{|x_{j,N}-t|}\leq\log\left(\frac{1}{\min\{\rho,|\gamma(\alpha_{i,N})-t|\}}\right),\qquad t\in V_{i,N}.
\]
If now $N$ is large enough so that $\kappa_{N}<\min\{\rho,1\}$, then $\mbox{diam}(V_{i,N})<\rho$ and therefore we have $|\gamma(\alpha_{i,N})-t|<\rho$ for all $t\in V_{i,N}$. So for all $N$ large enough such that $\kappa_{N}<\min\{\rho,1\}$, we have
\[
\log\frac{1}{\mathrm{diam}(K)}\leq \log\frac{1}{|x_{j,N}-t|}\leq\log\frac{1}{|\gamma(\alpha_{i,N})-t|},\qquad t\in V_{i,N},
\]
and using again $\mu_{\lambda}^{*}\leq R\lambda$ we obtain 
\begin{gather*}
\mu_{\lambda}^{*}(V_{i,N})\log\frac{1}{\mathrm{diam}(K)}\leq \int_{V_{i,N}}\log\frac{1}{|x_{j,N}-t|}\,d\mu_{\lambda}^{*}(t)\leq\int_{V_{i,N}}\log\frac{1}{|\gamma(\alpha_{i,N})-t|}\,d\mu_{\lambda}^{*}(t)\\
\leq R\int_{V_{i,N}}\log\frac{1}{|\gamma(\alpha_{i,N})-t|}\,d\lambda(t)
\leq R\sup_{x\in K}\int_{\{t\in K:|t-x|\leq \kappa_{N}\}}\log\frac{1}{|x-t|}\,d\lambda(t).
\end{gather*}
The same estimate is valid if $s_{j,N}\geq \beta_{i,N}$ (using $\beta_{i,N}$ instead of $\alpha_{i,N}$ in the argument). So from \eqref{specialassump} we deduce that
\begin{equation}\label{partialproof4}
\lim_{N\rightarrow\infty}\frac{1}{l_{N}}\sum_{1\leq |i-j|\leq 2}m_{j,N}\int_{V_{i,N}}\log\frac{1}{|x_{j,N}-t|}\,d\mu_{\lambda}^{*}(t)=0.
\end{equation}
We conclude from \eqref{partialproof3} and \eqref{partialproof4} that \eqref{asympvN2zero} holds.

So \eqref{partialproof2} will be justified if we prove
\begin{equation}\label{partialproof5}
\lim_{N\rightarrow\infty}\left(\frac{l_{N}}{N} u_{N,1}-v_{N,1}\right)=0.
\end{equation}
We have
\begin{equation}\label{partialproof6}
\left|\frac{l_{N}}{N} u_{N,1}-v_{N,1}\right|\leq \frac{1}{l_{N}}\sum_{j=1}^{l_{N}}m_{j,N}\Big(\sum_{i:|i-j|\geq 3}\Big|\frac{m_{i,N}}{N}\log\frac{1}{|x_{i,N}-x_{j,N}|}-\int_{V_{i,N}}\log\frac{1}{|x_{j,N}-t|}\,d\mu_{\lambda}^{*}(t)\Big|\Big).
\end{equation}

The intervals $\mathrm{Co}(\gamma^{-1}(V_{i,N}))$ and $\mathrm{Co}(\gamma^{-1}(V_{i+1,N}))$ are disjoint, they contain the points $s_{i,N}$ and $s_{i+1,N}$ respectively, and $s_{i,N}<s_{i+1,N}$. It follows that the intervals $W_{i,N}$ and $W_{i+1,N}$ have at most one endpoint in common and $W_{i,N}$ is located to the left of $W_{i+1,N}$. This implies $\beta_{i,N}\leq s_{i+1,N}$ and $s_{i,N}\leq \alpha_{i+1,N}$ for all $i$. Then $\beta_{i,N}\leq s_{i+1,N}<s_{i+2,N}\leq \alpha_{i+3,N}$ and therefore the intervals $W_{i,N}=[\alpha_{i,N},\beta_{i,N}]$ and $W_{i+3,N}=[\alpha_{i+3,N},\beta_{i+3,N}]$ are disjoint.

Suppose now that $|i-j|\geq 3$. Since the intervals $W_{i,N}$ and $W_{j,N}$ are disjoint, the function $\log\frac{1}{|x_{j,N}-\gamma(s)|}$ is continuous on the segment $W_{i,N}$. Therefore 
\[
\mu_{\lambda}^{*}(V_{i,N})\min_{s\in W_{i,N}} \log\frac{1}{|x_{j,N}-\gamma(s)|}\leq \int_{V_{i,N}}\log\frac{1}{|x_{j,N}-t|}\,d\mu_{\lambda}^{*}(t)\leq \mu_{\lambda}^{*}(V_{i,N})\max_{s\in W_{i,N}} \log\frac{1}{|x_{j,N}-\gamma(s)|}
\]
hence there exists $\xi_{i,N}\in W_{i,N}$ such that
\[
\int_{V_{i,N}}\log\frac{1}{|x_{j,N}-t|}d\mu_{\lambda}^{*}(t)=\mu_{\lambda}^{*}(V_{i,N})\log\frac{1}{|x_{j,N}-\gamma(\xi_{i,N})|}=\frac{m_{i,N}}{N}\log\frac{1}{|x_{j,N}-\gamma(\xi_{i,N})|}.
\]
Therefore we obtain
\begin{gather}
\Big|\frac{m_{i,N}}{N}\log\frac{1}{|x_{i,N}-x_{j,N}|}-\int_{V_{i,N}}\log\frac{1}{|x_{j,N}-t|}\,d\mu_{\lambda}^{*}(t)\Big|=\frac{m_{i,N}}{N}\left|\log\left(\frac{|x_{j,N}-\gamma(\xi_{i,N})|}{|x_{j,N}-x_{i,N}|}\right)\right|\notag\\
\leq \frac{R}{N}\left|\log\left(\frac{|x_{j,N}-\gamma(\xi_{i,N})|}{|x_{j,N}-\gamma(s_{i,N})|}\right)\right|.\label{keyineq1}
\end{gather}

Fix $1\leq j\leq l_{N}$. Let $I_{j}$ be the subinterval of $[a,b]$ such that $\gamma(I_{j})=B(x_{j,N},\rho)\cap \Gamma$. We define the following collections of indices:
\begin{align*}
\Delta_{j,1} & :=\{1\leq i\leq l_{N}: |i-j|\geq 3,\,\,\,W_{i,N}\subset I_{j}\},\\
\Delta_{j,2} & :=\{1\leq i\leq l_{N}: |i-j|\geq 3,\,\,\,W_{i,N}\not\subset I_{j}\}.
\end{align*}
Obviously $\Delta_{j,1}$ and $\Delta_{j,2}$ depend on $N$ but we omit this index in the notation. Suppose that $i\in \Delta_{j,1}$. Since $s_{i,N}, \xi_{i,N}\in W_{i,N}=[\alpha_{i,N},\beta_{i,N}]\subset I_{j}$ and $W_{i,N}$ is completely located either to the left or to the right of the point $s_{j,N}$, by the monotonicity of the function $r\mapsto |\gamma(r)-x_{j,N}|$, $r\in W_{i,N}$, we obtain
\begin{align}
\left|\log\left(\frac{|x_{j,N}-\gamma(\xi_{i,N})|}{|x_{j,N}-\gamma(s_{i,N})|}\right)\right| & \leq \log\left(\frac{|x_{j,N}-\gamma(\alpha_{i,N})|}{|x_{j,N}-\gamma(\beta_{i,N})|}\right),\qquad i\in \Delta_{j,1},\,\,\,i<j,\label{keyineq2}\\
\left|\log\left(\frac{|x_{j,N}-\gamma(\xi_{i,N})|}{|x_{j,N}-\gamma(s_{i,N})|}\right)\right| & \leq \log\left(\frac{|x_{j,N}-\gamma(\beta_{i,N})|}{|x_{j,N}-\gamma(\alpha_{i,N})|}\right),\qquad i\in \Delta_{j,1},\,\,\,i>j.\label{keyineq3}
\end{align}
So from \eqref{keyineq1}, \eqref{keyineq2}, and \eqref{keyineq3} we deduce that
\begin{gather*}
\sum_{i\in\Delta_{j,1}}\Big|\frac{m_{i,N}}{N}\log\frac{1}{|x_{i,N}-x_{j,N}|}-\int_{V_{i,N}}\log\frac{1}{|x_{j,N}-t|}\,d\mu_{\lambda}^{*}(t)\Big|\\
\leq \sum_{\substack{i\in\Delta_{j,1}\\i<j}}\frac{R}{N}\log\left(\frac{|x_{j,N}-\gamma(\alpha_{i,N})|}{|x_{j,N}-\gamma(\beta_{i,N})|}\right)+\sum_{\substack{i\in\Delta_{j,1}\\i>j}}\frac{R}{N}\log\left(\frac{|x_{j,N}-\gamma(\beta_{i,N})|}{|x_{j,N}-\gamma(\alpha_{i,N})|}\right).
\end{gather*}
If $i_{1}=\min\{i\in \Delta_{j,1},\,\,i<j\}$ and $i_{2}=\max\{i\in\Delta_{j,1},\,\,i<j\}$, then by the monotonicity of the logarithm we obtain
\[
\sum_{\substack{i\in\Delta_{j,1}\\i<j}}\log\left(\frac{|x_{j,N}-\gamma(\alpha_{i,N})|}{|x_{j,N}-\gamma(\beta_{i,N})|}\right)\leq\log\left(\frac{|x_{j,N}-\gamma(\alpha_{i_{1},N})|}{|x_{j,N}-\gamma(\beta_{i_{2},N})|}\right)\leq \log\left(\frac{\mathrm{diam}(\Gamma)N}{C}\right)
\]
where in the last inequality we used 
\[
|x_{j,N}-\gamma(\beta_{i_{2},N})|\geq |x_{j,N}-x_{j-1,N}|\geq \frac{C}{N},
\]
and $C$ is the constant in $\textrm{P3})$. In the same manner we can show that
\[
\sum_{\substack{i\in\Delta_{j,1}\\i>j}}\log\left(\frac{|x_{j,N}-\gamma(\beta_{i,N})|}{|x_{j,N}-\gamma(\alpha_{i,N})|}\right)\leq\log\left(\frac{\mathrm{diam}(\Gamma)N}{C}\right). 
\]
We conclude that
\begin{equation}\label{keyineq4}
\sum_{i\in\Delta_{j,1}}\Big|\frac{m_{i,N}}{N}\log\frac{1}{|x_{i,N}-x_{j,N}|}-\int_{V_{i,N}}\log\frac{1}{|x_{j,N}-t|}\,d\mu_{\lambda}^{*}(t)\Big|\leq \frac{2R}{N}\log\left(\frac{\mathrm{diam}(\Gamma)N}{C}\right).
\end{equation}
Note that this bound is valid for all $1\leq j\leq l_{N}$.   

Now we analyze the expression on the right-hand side of \eqref{keyineq1} for an index $i\in\Delta_{j,2}$. Let $i\in \Delta_{j,2}$, and let $N$ be large enough so that $\kappa_{N}<\rho/3$. Since $W_{i,N}\not\subset I_{j}$, it is easy to see that there exists a point $s_{i,N}^{*}\in W_{i,N}$ such that 
\begin{equation}\label{keyineq5}
|\gamma(s_{i,N}^{*})-x_{j,N}|\geq \rho.
\end{equation}
We have either $|x_{j,N}-\gamma(\xi_{i,N})|\geq |x_{j,N}-\gamma(s_{i,N})|$ or $|x_{j,N}-\gamma(\xi_{i,N})|<|x_{j,N}-\gamma(s_{i,N})|$. Suppose that the former inequality is true. Then applying the triangle inequality we obtain
\begin{gather}
\left|\log\left(\frac{|x_{j,N}-\gamma(\xi_{i,N})|}{|x_{j,N}-\gamma(s_{i,N})|}\right)\right|=\log\left(\frac{|x_{j,N}-\gamma(\xi_{i,N})|}{|x_{j,N}-\gamma(s_{i,N})|}\right)\notag\\
\leq \log\left(\frac{|x_{j,N}-\gamma(s_{i,N})|+|\gamma(s_{i,N})-\gamma(\xi_{i,N})|}{|x_{j,N}-\gamma(s_{i,N})|}\right)=\log\left(1+\frac{|\gamma(s_{i,N})-\gamma(\xi_{i,N})|}{|x_{j,N}-\gamma(s_{i,N})|}\right)\notag\\
\leq \frac{|\gamma(s_{i,N})-\gamma(\xi_{i,N})|}{|x_{j,N}-\gamma(s_{i,N})|}\label{keyineq6}
\end{gather}
where in the last step we used the inequality $\log(1+x)\leq x$ for $x\geq 0$. Since $\mathrm{diam}(V_{i,N})\leq \kappa_{N}<\rho/3$ and $x_{i,N}=\gamma(s_{i,N})\in V_{i,N}$, we have
\[
V_{i,N}\subset B(x_{i,N},\rho/3).
\]
Since $$\gamma(\alpha_{i,N}), \gamma(\beta_{i,N})\in \overline{V_{i,N}}\subset B(x_{i,N},\rho),$$ it follows that
\[
\gamma(\alpha_{i,N}), \gamma(\beta_{i,N})\in B(x_{i,N},\rho)\cap \Gamma=\gamma(I_{x_{i,N}}),
\]
therefore $\alpha_{i,N}, \beta_{i,N}\in I_{x_{i,N}}$, and so $W_{i,N}=[\alpha_{i,N},\beta_{i,N}]\subset I_{x_{i,N}}$. Since $\xi_{i,N}\in W_{i,N}$, by the monotonicity of the function $r\mapsto |x_{i,N}-\gamma(r)|$ for $r\in[\alpha_{i,N},s_{i,N}]$ and $r\in[s_{i,N},\beta_{i,N}]$ we obtain 
\begin{equation}\label{keyineq7}
|x_{i,N}-\gamma(\xi_{i,N})|\leq \max\{|x_{i,N}-\gamma(\alpha_{i,N})|,|x_{i,N}-\gamma(\beta_{i,N})|\}\leq \mathrm{diam}(\overline{V_{i,N}})\leq \kappa_{N}.
\end{equation}
Since $s_{i,N}^{*}\in W_{i,N}$, we also obtain the inequality
\begin{equation}\label{keyineq8}
|x_{i,N}-\gamma(s_{i,N}^{*})|\leq \kappa_{N}< \frac{\rho}{3}.
\end{equation}
It follows from \eqref{keyineq5} and \eqref{keyineq8} that
\[
|x_{j,N}-x_{i,N}|>\frac{2\rho}{3}
\]
which together with \eqref{keyineq6} and \eqref{keyineq7} gives the bound
\[
\left|\log\left(\frac{|x_{j,N}-\gamma(\xi_{i,N})|}{|x_{j,N}-\gamma(s_{i,N})|}\right)\right|\leq\frac{|\gamma(s_{i,N})-\gamma(\xi_{i,N})|}{|x_{j,N}-\gamma(s_{i,N})|}=\frac{|x_{i,N}-\gamma(\xi_{i,N})|}{|x_{j,N}-x_{i,N}|}\leq \frac{3 \kappa_{N}}{2\rho} 
\]
for $i\in\Delta_{j,2}$, in the case $|x_{j,N}-\gamma(\xi_{i,N})|\geq |x_{j,N}-\gamma(s_{i,N})|$.

Now assume that $i\in\Delta_{j,2}$, $\kappa_{N}<\rho/3$, and $|x_{j,N}-\gamma(\xi_{i,N})|<|x_{j,N}-\gamma(s_{i,N})|$. In this case we argue similarly, starting from the estimate
\begin{gather}
\left|\log\left(\frac{|x_{j,N}-\gamma(\xi_{i,N})|}{|x_{j,N}-\gamma(s_{i,N})|}\right)\right|=\log\left(\frac{|x_{j,N}-\gamma(s_{i,N})|}{|x_{j,N}-\gamma(\xi_{i,N})|}\right)\notag\\
\leq \log\left(\frac{|x_{j,N}-\gamma(\xi_{i,N})|+|\gamma(\xi_{i,N})-\gamma(s_{i,N})|}{|x_{j,N}-\gamma(\xi_{i,N})|}\right)=\log\left(1+\frac{|\gamma(s_{i,N})-\gamma(\xi_{i,N})|}{|x_{j,N}-\gamma(\xi_{i,N})|}\right)\notag\\
\leq \frac{|\gamma(s_{i,N})-\gamma(\xi_{i,N})|}{|x_{j,N}-\gamma(\xi_{i,N})|}.\label{keyineq9}
\end{gather}
The inequalities \eqref{keyineq5}, \eqref{keyineq7}, \eqref{keyineq8} are still valid. From \eqref{keyineq7} and \eqref{keyineq8} we deduce that 
\[
|\gamma(\xi_{i,N})-\gamma(s_{i,N}^{*})|\leq 2\kappa_{N}<2\rho/3
\]
which together with \eqref{keyineq5} implies that $|x_{j,N}-\gamma(\xi_{i,N})|>\rho/3$. So using \eqref{keyineq9} we obtain the estimate
\[
\left|\log\left(\frac{|x_{j,N}-\gamma(\xi_{i,N})|}{|x_{j,N}-\gamma(s_{i,N})|}\right)\right|\leq\frac{|\gamma(s_{i,N})-\gamma(\xi_{i,N})|}{|x_{j,N}-\gamma(\xi_{i,N})|}\leq \frac{3\kappa_{N}}{\rho}. 
\]
In conclusion, we have shown that for all $i\in\Delta_{j,2}$ and $N$ sufficiently large, we have
\[
\left|\log\left(\frac{|x_{j,N}-\gamma(\xi_{i,N})|}{|x_{j,N}-\gamma(s_{i,N})|}\right)\right|\leq \frac{3\kappa_{N}}{\rho}.
\]
From this inequality and \eqref{keyineq1} we obtain that for $N$ large enough,
\begin{equation}\label{keyineq10}
\sum_{i\in\Delta_{j,2}}\Big|\frac{m_{i,N}}{N}\log\frac{1}{|x_{i,N}-x_{j,N}|}-\int_{V_{i,N}}\log\frac{1}{|x_{j,N}-t|}\,d\mu_{\lambda}^{*}(t)\Big|\leq\sum_{i\in\Delta_{j,2}}\frac{R}{N}\frac{3\kappa_{N}}{\rho}\leq \frac{3R\,\kappa_{N}\, l_{N}}{N\rho}.
\end{equation}
This bound is valid for all $1\leq j\leq l_{N}$. So it follows from \eqref{partialproof6}, \eqref{keyineq4}, and \eqref{keyineq10} that 
\begin{align*}
\left|\frac{l_{N}}{N}u_{N,1}-v_{N,1}\right| & \leq\frac{1}{l_{N}}\sum_{j=1}^{l_{N}}m_{j,N}\left(\frac{2R}{N}\log\left(\frac{\mathrm{diam}(\Gamma)N}{C}\right)+\frac{3R\,\kappa_{N}\,l_{N}}{N\rho}\right)\\
& \leq R\left(\frac{2R}{N}\log\left(\frac{\mathrm{diam}(\Gamma)N}{C}\right)+\frac{3R\,\kappa_{N}\,l_{N}}{N\rho}\right) 
\end{align*}
for $N$ large enough. Since $l_{N}/N\rightarrow 1$ and $\kappa_{N}\rightarrow 0$ as $N\rightarrow\infty$, \eqref{partialproof5} follows.\qed

\begin{remark}
Observe that the hypotheses in Theorem~\ref{theo:main:upcon} regarding $\lambda$, $K$, and the partition sets $V_{j,N}$ are satisfied in Example~\ref{examplemain}.
\end{remark}

\begin{remark}\label{rem:altconfig}
A construction that is slightly different from \eqref{triangarray} but yields identical asymptotic results is as follows. Let 
\[
(y_{j,N})_{1\leq j\leq N},\qquad N\geq 2,
\] 
be a sequence of point configurations on $K$ satisfying the following properties:
\begin{itemize}
\item[1)] $\frac{1}{N}\sum_{j=1}^{N}\delta_{y_{j,N}}$ converges in the weak-star topology to a probability measure $\lambda$ on $K$.
\item[2)] For every $N$ there exists a partition $K=\bigcup_{j=1}^{N} V_{j,N}$ such that $y_{j,N}\in V_{j,N}$ and we have:
\begin{itemize}
\item[i)] $\lambda(V_{j,N})=1/N$ for every $1\leq j\leq N$.
\item[ii)] There exists a sequence $(\kappa_{N})$ of positive numbers that converges to zero such that $\mathrm{diam}(V_{j,N})\leq \kappa_N$ for all $1\leq j\leq N$ except possibly for a number $o(N)$ of indices $j$.
\end{itemize}
\item[3)] Properties $\textrm{P3)}$ and $\textrm{P4)}$ hold. 
\end{itemize}
If we define the constants 
\begin{equation}\label{def:altdN}
d_{N}'=\inf\{E_{N,f}(y_{1,N},\ldots,y_{N,N},m_{1},\ldots,m_{N}): m_{1},\ldots,m_{N}\in I\},
\end{equation}
then it is not difficult to see that under the same assumptions of Theorem~\ref{theo:main:upcon}, the sequences $d_{N}'$ and $\frac{1}{N}\sum_{j=1}^{N}m_{j,N}\,\delta_{y_{j,N}}$ with optimal masses $m_{j,N}$ for \eqref{def:altdN} converge to the limits indicated in that theorem.
\end{remark}

\bigskip

\noindent\textbf{Acknowledgements:} We are grateful to the referee for his/her thoughtful comments which have resulted in an improvement of the content of this paper. In particular, for providing the proof of Lemma \ref{lemreferee} which allowed us to eliminate a redundant assumption in a previous version of Theorem \ref{theo:main:upcon}. A. Tovbis was partially supported by the NSF grant DMS-2009647.

\bigskip

\noindent \textsc{Department of Mathematics, University of Central Florida, 4393 Andromeda Loop North, Orlando, FL 32816, USA} \\
\textit{Email address}: \texttt{abey.lopez-garcia\symbol{'100}ucf.edu}

\bigskip

\noindent \textsc{Department of Mathematics, University of Central Florida, 4393 Andromeda Loop North, Orlando, FL 32816, USA} \\
\textit{Email address}: \texttt{alexander.tovbis\symbol{'100}ucf.edu}

\begin{thebibliography}{99}

\bibitem{BeckGry}
B. Beckermann and A. Gryson, Extremal rational functions on symmetric discrete sets and superlinear convergence of the ADI method, Constr. Approx. 32 (2010), 393--428.

\bibitem{BeckKuij}
B. Beckermann and A.B.J. Kuijlaars, Superlinear convergence of conjugate gradients, SIAM J. Numer. Anal. 39 (2001), no. 1, 300--329.

\bibitem{BCC}
C. Beltr\'{a}n, N. Corral, and J.G. Criado del Rey, Discrete and continuous green energy on compact manifolds, J. Approx. Theory 237 (2019), 160--185. 

\bibitem{BHS}
S. Borodachov, D.P. Hardin, and E.B. Saff, \textit{Discrete Energy on Rectifiable Sets}, Springer Monographs in Mathematics, 2019, Springer. 

\bibitem{BrauDragSaff}
J.S. Brauchart, P.D. Dragnev, and E.B. Saff, Riesz external field problems on the hypersphere and optimal point separation, Potential Anal. 41 (2014), 647--678.

\bibitem{DragOriSaffWiel}
P.D. Dragnev, R. Orive, E.B. Saff, and F. Wielonsky, Riesz energy problems with external fields and related theory, Constr. Approx. 57 (2023), 1--43.
 
\bibitem{DragnevSaff}
P.D. Dragnev and E.B. Saff, Constrained energy problems with applications to orthogonal polynomials of a discrete variable, J. d'Analyse Math. 72 (1997), 223--259.

\bibitem{KuijTov}
A. Kuijlaars and A. Tovbis, On minimal energy solutions to certain classes of integral equations related to soliton gases for integrable systems, Nonlinearity 34 (2021), 7227--7254.

\bibitem{KuijVan}
A.B.J. Kuijlaars and W. Van Assche, Extremal polynomials on discrete sets, Proc. London Math. Soc. 79 (1999), 191--221.

\bibitem{Landkof}
N.S. Landkof, \textit{Foundations of Modern Potential Theory}, Grundlehren der mathematischen Wissenschaften, Vol. 180, Springer-Verlag, 1972. 

\bibitem{Lieb}
E.H. Lieb, Thomas-Fermi and related theories of atoms and molecules, Rev. Mod. Phys. 53, 4 (1981), 603--641.

\bibitem{Lop}
A. L\'{o}pez Garc\'{i}a, Greedy energy points with external fields,  in \textit{Recent Trends in Orthogonal Polynomials and Approximation Theory} (J. Arves\'{u}, F. Marcell\'{a}n, and A. Mart\'{i}nez-Finkelshtein, eds.), Contemp. Math. 507, 2010, 189--207. 

\bibitem{NikSor}
E.M. Nikishin and V.N. Sorokin, \textit{Rational Approximations and Orthogonality}, Transl. Math. Monographs, Vol. 92, Amer. Math. Soc., 1991.

\bibitem{Rakh}
E.A. Rakhmanov, Equilibrium measure and the distribution of zeros of the extremal polynomials of a discrete variable, Mat. Sb. 187 (8) (1996), 109--124 (in Russian); English transl. in Sb. Math. 187 (8) (1996), 1213--1228. 

\bibitem{Rans}
T. Ransford, \textit{Potential Theory in the Complex Plane}, London Math. Soc. Student Texts, Vol. 28, Cambridge University Press, 1995. 

\bibitem{SaffTotik}
E.B. Saff and V. Totik, \textit{Logarithmic Potentials with External Fields}, Grundlehren der mathematischen Wissenschaften, Vol. 316, Springer-Verlag, 1997.

\bibitem{TovWang}
A. Tovbis and F. Wang, Recent developments in spectral theory of the focusing NLS soliton and breather gases: the thermodynamic limit of average densities, fluxes and certain meromorphic differentials; periodic gases, J. Phys. A: Math. Theor. 55 (2022), 424006.

\bibitem{Zorii1}
N.V. Zorii, On the solvability of the Gauss variational problem, Comput. Meth. Funct. Theory 2 (2002), 427--448.

\bibitem{Zorii2}
N.V. Zorii, Equilibrium potentials with external fields, Ukrainian Math. J. 55 (2003), 1423--1444.

\bibitem{Zorii3}
N.V. Zorii, Equilibrium problems for infinite dimensional vector potentials with external fields, Potential Anal. 38 (2013), 397--432.

\end{thebibliography}
\end{document}